\documentclass[12pt]{amsart}
\usepackage[francais,english]{babel}
\usepackage[T1]{fontenc}
\usepackage[latin1]{inputenc}
\usepackage{amsmath,amssymb,amsthm,mathrsfs}
\usepackage[all,2cell]{xy}
\usepackage{graphicx}
\usepackage{enumerate} 
\usepackage{lmodern}
\usepackage{mathrsfs} 
\usepackage{dsfont}
\usepackage{color}

\DeclareMathOperator{\torh}{tor_\hbar}
\DeclareMathOperator{\Ann}{Ann}
\DeclareMathOperator{\Tor}{Tor}

\DeclareMathOperator{\RHom}{RHom}
\DeclareMathOperator{\Ext}{Ext}
\DeclareMathOperator{\fRHom}{R\mathcal{H}om}

\DeclareMathOperator{\id}{id}
\DeclareMathOperator{\opp}{op}
\DeclareMathOperator{\Mod}{Mod}

\DeclareMathOperator{\Hn}{H}

\DeclareMathOperator{\gr}{gr}
\DeclareMathOperator{\Rg}{R\Gamma}
\DeclareMathOperator{\dR}{R}

\DeclareMathOperator{\Supp}{Supp}
\DeclareMathOperator{\supp}{supp}
\DeclareMathOperator{\codim}{codim}

\DeclareMathOperator{\Ker}{Ker}
\DeclareMathOperator{\im}{Im}

\begin{document}

\theoremstyle{plain} 
\newtheorem{thm}{Theorem}[section]
\newtheorem{cor}[thm]{Corollary}
\newtheorem{prop}[thm]{Proposition}
\newtheorem{lemme}[thm]{Lemma}
\newtheorem{conj}[thm]{Conjecture}
\newtheorem*{theoetoile}{Theorem} 
\newtheorem*{conjetoile}{Conjecture} 
\newtheorem*{theoetoilefr}{Theorem}
\newtheorem*{propetoilefr}{Proposition}

\theoremstyle{definition} 
\newtheorem{defi}[thm]{Definition}
\newtheorem{example}[thm]{Example}
\newtheorem{examples}[thm]{Examples}
\newtheorem{question}[thm]{Question}
\newtheorem{Rem}[thm]{Remark}
\newtheorem{Notation}[thm]{Notation}

\numberwithin{equation}{section}

\newcommand{\On}[1]{\mathcal{O}_{#1}}
\newcommand{\En}[1]{\mathcal{E}_{#1}}
\newcommand{\Fn}[1]{\mathcal{F}_{#1}} 
\newcommand{\tFn}[1]{\mathcal{\tilde{F}}_{#1}}
\newcommand{\hum}[1]{hom_{\mathcal{A}}({#1})}
\newcommand{\hcl}[2]{#1_0 \lbrack #1_1|#1_2|\ldots|#1_{#2} \rbrack}
\newcommand{\hclp}[3]{#1_0 \lbrack #1_1|#1_2|\ldots|#3|\ldots|#1_{#2} \rbrack}
\newcommand{\catMod}{\mathsf{Mod}}
\newcommand{\Der}{\mathsf{D}}
\newcommand{\Ds}{D_{\mathbb{C}}}
\newcommand{\DG}{\mathsf{D}^{b}_{dg,\mathbb{R}-\mathsf{C}}(\mathbb{C}_X)}
\newcommand{\lI}{[\mspace{-1.5 mu} [}
\newcommand{\rI}{] \mspace{-1.5 mu} ]}
\newcommand{\Ku}[2]{\mathfrak{K}_{#1,#2}}
\newcommand{\iKu}[2]{\mathfrak{K^{-1}}_{#1,#2}}
\newcommand{\Be}{B^{e}}
\newcommand{\op}[1]{#1^{\opp}}
\newcommand{\N}{\mathbb{N}}
\newcommand{\Ab}[1]{#1/\lbrack #1 , #1 \rbrack}
\newcommand{\Du}{\mathbb{D}}
\newcommand{\C}{\mathbb{C}}
\newcommand{\Z}{\mathbb{Z}}
\newcommand{\w}{\omega}
\newcommand{\K}{\mathcal{K}}
\newcommand{\Hoc}{\mathcal{H}\mathcal{H}}
\newcommand{\env}[1]{{\vphantom{#1}}^{e}{#1}}
\newcommand{\eA}{{}^eA}
\newcommand{\eB}{{}^eB}
\newcommand{\eC}{{}^eC}
\newcommand{\cA}{\mathcal{A}} 
\newcommand{\cB}{\mathcal{B}}
\newcommand{\cR}{\mathcal{R}}
\newcommand{\cL}{\mathcal{L}}
\newcommand{\cO}{\mathcal{O}}
\newcommand{\cM}{\mathcal{M}}
\newcommand{\cN}{\mathcal{N}}
\newcommand{\cK}{\mathcal{K}}
\newcommand{\cC}{\mathcal{C}}
\newcommand{\cF}{\mathcal{F}}
\newcommand{\cG}{\mathcal{G}}
\newcommand{\cP}{\mathcal{P}}
\newcommand{\cQ}{\mathcal{Q}}
\newcommand{\cU}{\mathcal{U}}
\newcommand{\cE}{\mathcal{E}}
\newcommand{\Hper}{\Hn^0_{\textrm{per}}}
\newcommand{\Dper}{\Der_{\mathrm{perf}}}
\newcommand{\Yo}{\textrm{Y}}
\newcommand{\gqcoh}{\mathrm{gqcoh}}
\newcommand{\coh}{\mathrm{coh}}
\newcommand{\cc}{\mathrm{cc}}
\newcommand{\qcc}{\mathrm{qcc}}
\newcommand{\gd}{\mathrm{gd}}
\newcommand{\qcoh}{\mathrm{qcoh}}
\newcommand{\lcl}{\mathrm{lcl}}
\newcommand{\fin}{\mathrm{fin}}
\newcommand{\obplus}[1][i \in I]{\underset{#1}{\overline{\bigoplus}}}
\newcommand{\Lte}{\mathop{\otimes}\limits^{\rm L}}
\newcommand{\te}{\mathop{\otimes}\limits^{}}
\newcommand{\pt}{\textnormal{pt}}
\newcommand{\A}[1][X]{\cA_{{#1}}}
\newcommand{\dA}[1][X]{\cC_{X_{#1}}}
\newcommand{\conv}[1][]{\mathop{\circ}\limits_{#1}}
\newcommand{\sconv}[1][]{\mathop{\ast}\limits_{#1}}
\newcommand{\reim}[1]{\textnormal{R}{#1}_!}
\newcommand{\roim}[1]{\textnormal{R}{#1}_\ast}
\newcommand{\ldetens}{\overset{\mathnormal{L}}{\underline{\boxtimes}}}
\newcommand{\br}{\bigr)}
\newcommand{\bl}{\bigl(}
\newcommand{\sC}{\mathscr{C}}
\newcommand{\ucat}{\mathbf{1}}
\newcommand{\ubtimes}{\underline{\boxtimes}}
\newcommand{\uLte}{\mathop{\underline{\otimes}}\limits^{\rm L}} 
\newcommand{\Lp}{\mathrm{L}p}
\newcommand{\pder}[3][]{\frac{\partial^{#1}#2}{\partial{#3}}}
\newcommand{\reg}{\mathrm{reg}}
\newcommand{\sing}{\mathrm{sing}}
\newcommand{\fExt}{\mathcal{E}xt}
\newcommand{\dL}{\mathrm{L}}
\newcommand{\fgd}{\mathrm{fgd}}

\author{Fran\c{c}ois Petit}
\email{francois.petit@uni.lu}
\thanks{The author was supported by the EPSRC grant EP/G007632/1}
\address{University of Luxembourg\\ 
Campus Kirchberg, FSTC\\ 
Mathematics Research Unit\\ 
6, rue Coudenhove-Kalergi\\
L-1359 Luxembourg-Kirchberg\\
Grand-Duchy of Luxembourg.}
\title[]{The Codimension-Three conjecture for holonomic DQ-modules}
\begin{abstract} 
We prove an analogue for holonomic DQ-modules of the codimension-three conjecture for microdifferential modules recently proved by Kashiwara and Vilonen.  Our result states that any holonomic DQ-module having a lattice extends uniquely beyond an analytic subset of codimension equal to or larger than three in a Lagrangian subvariety containing the support of the DQ-module. 
\end{abstract}

\maketitle
\section{Introduction}
Problems of extension of analytic objects have been influential in complex analysis and complex analytic geometry. In 1966, Serre asked the following question in \cite{Serre}. Let $X$ be an analytic space and let $S$ be a closed analytic subset of $X$.
\begin{center}
\textit{ Si $\codim S \geq 3$, est-il vrai que tout faisceau localement libre (ou même seulement réflexif) sur $X-S$ est prolongeable?} \footnote{If $\codim S \geq 3$, is it true that any locally free (or even only reflexive) sheaf on $X - S$ is extendable ?}
\end{center}
This question stimulated intense activity around problems of extension of coherent analytic sheaves. One of the major results in this area is an extension theorem due to Trautmann, Frisch-Guenot, and Siu (see \cite{FG, Siu, T}) answering Serre's question.

\begin{thm}[Frisch-Guenot, Trautmann, and Siu]\label{thm:FGTS}
Let $X$ be a complex manifold, let $S$ be a closed analytic subset of $X$ and $j:X \setminus S \to X$ be the open embedding of $X \setminus S$ into $X$. If $\cF$ is a reflexive coherent $\cO_{X \setminus S}$-module and $\codim S \geq 3$ then $j_\ast\cF$ is a coherent $\cO_X$-module.
\end{thm}

In the 1970's, Sato, Kashiwara and Kawai introduced and studied systematically the theory of microdifferential systems (see \cite{SKK}). The codimension-three conjecture was formulated by Kashiwara at the end of the 1970's and based on his study of holonomic systems with regular singularities. This conjecture is concerned with the extension of a holonomic microdifferential system over an analytic subset of the cotangent bundle of a complex manifold. It was recently proved by Kashiwara and Vilonen (see \cite{KV}).
\begin{thm}[{\cite[Theorem 1.2]{KV}}]
Let $X$ be a complex manifold, $U$ an open subset of $T^\ast X$, $\Lambda$ a closed Lagrangian analytic subset of $U$, and $Y$ a closed analytic subset of $\Lambda$ such that $\codim_\Lambda Y \geq 3$. Let $\cE_X$ the sheaf of microdifferential operators on $T^\ast X$ and $\cM$ be a holonomic $(\mathcal{E}_X|_{U \setminus Y})$-module whose support is contained in $\Lambda \setminus Y$. Assume that $\cM$ possesses an $(\mathcal{E}_X(0)|_{U \setminus Y})$-lattice. Then $\cM$ extends uniquely to a holonomic module defined on $U$ whose support is contained in $\Lambda$.
\end{thm} 

The proof of the conjecture was made possible by the deep result of Kashiwara and Vilonen extending Theorem \ref{thm:FGTS} to coherent sheaves over $\cO_X^\hbar:=\cO_X[[\hbar]]$.
\begin{thm}[{\cite[Theorem 1.6]{KV}}]\label{thm:codimcom}
Let X be a complex manifold, let $Y$ be a closed analytic subset of $X$ and $j:X \setminus Y \to X$ the open embedding of $X \setminus Y$ into $X$. If $\cN$ is a coherent reflexive $\cO^\hbar_{X \setminus Y}$-module and $\codim Y \geq 3$ then $j_\ast\cN$ is a coherent $\cO_X^\hbar$-module.
\end{thm}
The techniques involved in the proof of this result are inspired by ideas coming from deformation quantization (see \cite{KS3}).

Deformation Quantization modules (DQ-modules for short) on symplectic manifolds provide a generalization of microdifferential modules to arbitrary symplectic manifolds (see \cite{KS3}). It is therefore natural to try to extend the codimension-three conjecture to holonomic DQ-modules, and that is what we do here. 

The case of DQ-modules is strictly more general than the case of microdifferential modules. Indeed, in forthcoming work, we will explain how to recover the codimension-three conjecture for formal microdifferential operators from the statement for DQ-modules (the converse is not true). Although, our result is more general and does not depend on the codimension-three conjecture for microdifferential modules, it relies fundamentally on the tools elaborated by Kashiwara and Vilonen and our proof follows their general strategy. The main difference between the two approaches is that they are working in a conical setting whereas we are not. Hence, we have to implement their strategy differently. 

One of the important aspects of the codimension-three conjecture for microdifferential modules is that it provides essential information on the structure of the stack of microlocal perverse sheaves through the microlocal Riemann-Hilbert correspondence (see \cite{A1}, \cite{A2}, \cite{W}). However, the interpretation of the codimension-three conjecture for holonomic DQ-modules in terms of perverse sheaves is not yet clear since the analogue of microlocal perverse sheaves is not known for DQ-modules. It is an interesting problem to define such objects.

Here are the precise statements of the results we are proving.
\begin{thm}\label{thm:codim}
Let $X$ be a complex manifold endowed with a DQ-algebroid stack $\cA_X$ such that the associated Poisson structure is symplectic. Let $\Lambda$ be a closed Lagrangian analytic subset of $X$ and $Y$ a closed analytic subset of $\Lambda$ such that $\codim_\Lambda Y \geq 3$. Let $\cM$ be a holonomic $(\cA^{loc}_X|_{X \setminus Y})$-module, whose support is contained in $\Lambda \setminus Y$. Assume that $\cM$ has an $\cA_X|_{X \setminus Y}$-lattice. Then $\cM$ extends uniquely to a holonomic module defined on $X$ whose support is contained in $\Lambda$. 
\end{thm}
\noindent and
\begin{thm} \label{thm:subcodim}
Let $X$ be a complex manifold endowed with a DQ-algebroid stack $\cA_X$ such that the associated Poisson structure is symplectic. Let $\Lambda$ be a closed Lagrangian analytic subset of $X$ and $Y$ a closed analytic subset of $\Lambda$ such that $\codim_\Lambda Y \geq 2$. Let $\cM$ be a holonomic $\cA_X^{loc}$-module whose support is contained in $\Lambda$ and let $\cM_1$ be an $\cA_{X}^{loc}|_{X \setminus Y}$-submodule of $\cM|_{X \setminus Y}$. Then $\cM_1$ extends uniquely to a holonomic $\cA_X^{loc}$-submodule of $\cM$.
\end{thm}

We hope that the submodule version of the codimension-three conjecture for DQ-modules may have applications to the representation theory of Cherednik algebras, via the localization theorem due to Kashiwara and Rouquier (see \cite{KR}).

Let us describe briefly the proof of the above statements. As we already remarked, we follow the general strategy of Kashiwara and Vilonen (see \cite{KV}). Keeping the notations of the theorems \ref{thm:codim} and \ref{thm:subcodim} and denoting by $j:X \setminus Y \to X$ the open embedding of $X\setminus Y$ into $X$, the problem essentially amount to show that $j_\ast \cM$ and $j_\ast \cM_1$ are coherent. This is a local question. By adapting a standard technique in several complex variables in \cite{SKK} to the case of DQ-modules, proving the coherence of a DQ-module amounts to proving the coherence of the pushforward of this module by a certain projection, the restriction of which to the support of the module is finite. Taking the pushforward by the aforementioned projection reduce the non-commutative problem to a commutative one. Then it is possible to use Theorem \ref{thm:codimcom} to establish the coherence of the direct image of $j_\ast \cM$ and $j_\ast \cM_1$ by this specific projection.

This paper is organised as follows. In the second section, we review some notions of the theory of DQ-modules and prove a few facts concerning holonomic DQ-modules that we will need later on. In the third section, we adapt the coherence criterion from \cite{SKK} corresponding to Proposition 4.1 and 4.4 of \cite{KV} to the case of DQ-modules. In \cite{KV}, it is indicated that the statements for microdifferential modules can be proved by using the Weierstrass preparation and division theorems for microdifferential operators. In the case of DQ-modules, we follow a different route. We first recall the corresponding statement for coherent analytic sheaves and then lift these statements to the world of DQ-modules. In the fourth section, we prove the analogue of \cite[Proposition 4.2]{KV} for DQ-modules. The proof is significantly more complicated than in the case of microdifferential modules since the duality theory for DQ-modules is entirely stated in terms of kernels and not of morphisms. Thus, one has to identify the actions of all the kernels involved in the duality with their corresponding operations on sheaves. In the short fifth section, we adapt the reduction technique of \cite{SKK} to a non-conic setting and prove the version of the codimension-three conjecture for holonomic DQ-modules in the sixth section. Finally, in the last section, we present a conjecture due to Pierre Schapira which extends the codimension-three conjecture for holonomic DQ-modules to all Poisson manifolds.\\

\noindent{\textbf{Acknowledgments.}} I would like to express my gratitude to Masaki Kashiwara and Pierre Schapira for many enlightening explanations. It is a pleasure to thanks Iain Gordon for his support and many useful discussions as well as Michael Wemyss, Thibault Lemanissier, Chris Dodd, Gwyn Bellamy and Evgeny Shinder for useful conversation and Will Donovan for proof-reading part of this paper.

\section{DQ-modules}
In this section, we recall and prove some results concerning the general theory of DQ-modules.

We denote by $\C^\hbar:=\C[[\hbar]]$ the ring of formal power series with coefficients in $\C$ and by $\C^{\hbar,loc}:=\C((\hbar))$ the field of Laurent series with complex coefficients.  In all this section, $(X,\cO_X)$ is a complex manifold.

\subsection{DQ-algebras}
In this subsection, we review some facts concerning DQ-algebras. 

We define the following sheaf of $\C^\hbar$-algebras
\begin{equation*}
\cO_X^\hbar:=\varprojlim_{n \in \N} \cO_X \te_\C (\C^\hbar / \hbar^n \C^\hbar).
\end{equation*}

\begin{defi}
A star-product denoted $\star$ on $\cO_X^\hbar$ is a $\C^\hbar$-bilinear multiplicative law satisfying
\begin{equation*}
f \star g = \sum_{i \geq 0} P_i(f,g) \hbar^i \;\; \textnormal{for every} \;f, \; g \in \cO_X,
\end{equation*}
where the $P_i$ are bi-differential operators such that for every $f, g \in \cO_X, P_0(f,g)=fg$ and 
$P_i(1,f)=P_i(f,1)=0$ for $i>0$. The pair $(\cO_X^\hbar, \star)$ is called a star-algebra. 
\end{defi}

\begin{example}\label{ex:Moyal}
Let $U$ be an open subset of $T^\ast\C^{n}$ endowed with a symplectic coordinate system $(x;u)$ with $x=(x_1,\ldots,x_n)$ and  $u=(u_1,\ldots,u_n)$. Then there is  a star-algebra $(\cO_{U} ^\hbar, \star)$ on $U$ given by
\begin{equation*}
f \star g = \sum_{\alpha \in \N^n} \dfrac{\hbar^{|\alpha|}}{\alpha !} (\partial^\alpha_u f) (\partial^\alpha_x g).
\end{equation*}
This star-product is called the Moyal-Weyl star-product.
\end{example}

\begin{Rem}\label{rem:reduction}
On a symplectic manifold all the star-algebras are locally isomorphic.
\end{Rem}

\begin{defi}
\noindent (i) A DQ-algebra $\cA_X$ on $X$ is a $\C^\hbar$-algebra locally isomorphic to a star-algebra as a $\C^\hbar$-algebra.

\noindent (ii) For a DQ-algebra $\cA_X$, we set $\cA_X^{loc}:=\C^{\hbar, loc} \te_{\C^\hbar} \cA_X$.
\end{defi}

There is a unique $\C$-algebra isomorphism $\cA_X / \hbar \cA_X \stackrel{\sim}{\longrightarrow} \cO_X$. We write $\sigma_0: \cA_X \twoheadrightarrow \cO_X$ for the epimorphism of $\C$-algebras defined by
\begin{equation*}
\cA_X \to \cA_X / \hbar \cA_X \stackrel{\sim}{\longrightarrow} \cO_X.
\end{equation*}
This induces a Poisson bracket $\lbrace \cdot,\cdot \rbrace$ on $\cO_X$ defined as follows:
\begin{equation*}
\textnormal{for every $a, \; b \in \cA_X$}, \; \lbrace \sigma_0(a),\sigma_0(b) \rbrace=\sigma_0(\hbar^{-1}(ab-ba)).
 \end{equation*}
\begin{Notation}
If $\cA_X$ is a DQ-algebra, we denote by $\cA_{X^a}$ the opposite algebra $\cA_X^{\opp}$. 
If $X$ and $Y$ are two manifolds endowed with DQ-algebras $\cA_X$ and $\cA_Y$, then $X \times Y$ is canonically endowed with the DQ-algebra $\cA_{X \times Y}:=\cA_X \underline{\boxtimes} \cA_Y$ (see \cite[§2.3]{KS3}).
\end{Notation}
\subsection{DQ-modules}

We refer the reader to \cite{KS3} for an in-depth study of DQ-modules. In this section, we briefly recall some general results that we will subsequently use.

Let $(X,\cO_X)$ be a complex manifold endowed with a DQ-algebroid stack $\cA_X$. We denote by $\Mod(\cA_X)$ the Grothendieck category of modules over $\cA_X$, by $\Der(\cA_X)$ its derived category, by $\Mod_\coh(\cA_X)$ the Abelian category of coherent modules over $\cA_X$ and by $\Der^b_{\coh}(\cA_X)$ the bounded derived category of DQ-modules with coherent cohomology and if $Z$ is a closed analytic subset of $X$ we denote by $\Der^b_{\coh, Z}(\cA_X)$ the full subcategory of $\Der^b_{\coh}(\cA_X)$ the objects of which are supported in $Z$. We use similar notations for $\cA_X^{loc}$-modules.

If $\cA_X$ is a DQ-algebroid stack one can associate to it an algebroid stack $\gr_\hbar \cA_X$ (see \cite[p.69]{KS3}). If $\cA_X$ is a DQ-algebra, then $\gr_\hbar \cA_X \simeq \cO_X$. This provides a functor
\begin{equation*}
\gr_\hbar:\Der(\cA_X) \to \Der(\gr_\hbar \cA_X), \; \cM \mapsto \C \Lte_{\C^\hbar} \cM.
\end{equation*}
We denote by $for$ the forgetful functor $for:\Mod(\cA_X^{loc}) \to \Mod(\cA_X)$ the left adjoint of which is given by
\begin{align*}
(\cdot)^{loc}:& \Mod(\cA_X) \to \Mod(\cA_X^{loc})\\
              & \cN \mapsto \cA_X^{loc} \te_{\cA_X} \cN .
\end{align*}

\begin{defi}
Let $\cM \in \Mod(\cA_X)$. We say that
\begin{enumerate}[(i)]
\item $\cM$ has no $\hbar$-torsion if the map $\cM \stackrel{\hbar}{\to}\cM$ is a monomorphism,
\item the module $\cM$ is of uniform $\hbar$-torsion if there exists $k \in \N$ such that  $\hbar^k \cM=(0)$. The smallest such a $k \in \N$ is called the index of $\hbar$-torsion of $\cM$ and is denoted $\torh(\cM)$,
\item $\cM$ is of $\hbar$-torsion if it is locally of uniform $\hbar$-torsion.
\end{enumerate}
\end{defi}
Finally, we recall the following lemma.
\begin{lemme}
Let $\cM \in \Mod_{\coh}(\cA_X)$. Then $\cM$ is the extension of a module without $\hbar$-torsion by  a $\hbar$-torsion module.
\end{lemme}

\subsubsection{Duality theory for DQ-modules}
 We briefly recall the main features of the duality theory for DQ-modules and refer the reader to \cite[Ch. 6]{KS3} for a detailed study.

Let X be a complex manifold endowed with a DQ-algebroid stack $\cA_X$.
We denote by $\delta: X \hookrightarrow X \times X$ the diagonal embedding of $X$ into $X \times X$ and set $\cC_X:=\delta_\ast \cA_X$. The object $\cC_X$ is an $\cA_{X \times X^a}$-module simple along the diagonal. 

We denote by $\w_X$ the dualizing complex for $\cA_X$-modules. It is a bi-invertible $(\cA_X \otimes \cA_{X^a})$-module. Since the category of bi-invertible $(\cA_X \otimes \cA_{X^a})$-modules is equivalent to the category of $\cA_{X \times X^a}$-modules simple along the diagonal we will regard $\w_X$ as an $\cA_{X \times X^a}$-module simple along the diagonal and will still denote in by $\w_X$. By \cite[Theorem 6.2.4]{KS3}, we have%
\begin{lemme}
Let $X$ be a complex symplectic manifold. There is a canonical isomorphism 
\begin{equation*}
\w_X \simeq \cC_X[d_X]
\end{equation*}
of $\cA_{X \times X^a}$-modules.
\end{lemme}
We set
\begin{align*}
\Du^\prime_{\cA_X}:&\Der(\cA_X)^{\opp} \to \Der(\cA_X)\\
 & \cM \mapsto \fRHom_{\cA_X}(\cM, \cA_{X}). 
\end{align*}
We use a similar notation for $\cA_X^{loc}$-modules.
\begin{Notation}
\begin{enumerate}[(i)]
\item
Consider a product of manifolds $X_1\times X_2 \times X_3$, we write it $X_{123}$. We denote by
$p_i$ the $i$-th projection and by $p_{ij}$ the $(i,j)$-th projection
({\em e.g.,} $p_{13}$ is the projection from 
$X_1\times X_1^a\times X_2$ to $X_1\times X_2$). 
\item We write $\A[i]$ and $\A[ij^a]$
instead of $\A[X_i]$ and $\A[X_i\times X_j^a]$  and
similarly with other products. We use the same notations 
for $\dA[_i]$.
\end{enumerate}
\end{Notation}%
Recall Definitions 3.1.2 and 3.1.3 of \cite{KS3}.
\begin{defi}
Let $\cK_i\in\Der^b(\A[ij^a])$ ($i=1,2$, $j=i+1$). 
One sets
\begin{align*}
\cK_1 \uLte_{\cA_2} \cK_2 &= 
(\cK_1\ldetens\cK_2)\Lte_{\A[22^a]} \dA[2]\\
&=p_{12}^{-1} \cK_1 \Lte_{p_{12}^{-1}\cA_{1^a2}} \cA_{123} \Lte_{p_{23^a}^{-1} \cA_{23^a}} p_{23}^{-1} \cK_2,\\
\cK_1\conv[X_2]\cK_2 &= 
\reim{p_{13}}\bl\cK_1 \uLte_{\cA_2} \cK_2\br,\\
\cK_1\sconv[X_2]\cK_2&= 
\roim{p_{13}}\bl\cK_1 \uLte_{\cA_2} \cK_2\br.
\end{align*}
\end{defi}
\begin{Rem}
\noindent(a) There is a morphism $\cK_1 \Lte_{\cA_2} \cK_2 \to \cK_1 \uLte_{\cA_2} \cK_2$ which is an isomorphism if $X_1=\pt$ or $X_3=\pt$.

\noindent(b) If the projection $p_{13}:\Supp(\cK_1) \times_{X_2} \Supp(\cK_2) \to X_1 \times X_3$ is proper then, $\cK_1\conv[X_2]\cK_2 \simeq \cK_1\sconv[X_2]\cK_2$.  
\end{Rem}
We have the following duality results.
\begin{thm}[{\cite[Theorem 3.3.6]{KS3}}]\label{thm:dualite}
Let $(X_i, \cA_i)$ $(i=1,2,3)$ be a complex manifold endowed with a DQ-algebroid $\cA_{i}$. Let $Z_i$ be a closed subset of $X_i \times X_{i+1}$ and assume that $Z_1 \times_{X_2} Z_2$ is proper over $X_1 \times X_3$. Set $Z=p_{13} (p_{12}^{-1} Z_1 \cap p_{23}^{-1} Z_2)$. Let $\cK_i \in \Der_{\coh, Z_i}^b(\cA_{i{(i+1)}^a})$  $(i=1,2)$. Then the object $\cK_1 \underset{X_2}{\circ} \cK_2$ belongs to $\Der_{\coh, Z}^b(\cA_{13^a})$  and we have a natural isomorphism
\begin{equation*}
\Du_{\cA_{1{2^a}}}^\prime(\cK_1) \underset{X_{2^a}}{\circ} \omega_{X_{2^a}} \underset{X_{2^a}}{\circ} \Du_{\cA_{2{3^a}}}^\prime(\cK_2) \stackrel{\sim}{\to} \Du_{\cA_{1{3^a}}}^\prime(\cK_1 \underset{X_2}{\circ} \cK_2).
\end{equation*}
\end{thm}

\begin{defi}
\begin{enumerate}[(i)]
\item An $\cA_X$-module $\cN$ is a submodule of an $\cA_X^{loc}$-module $\cM$ if there exist a monomorphism $\mu:\cN \to for(\cM)$ in $\Mod(\cA_X)$.

\item Let $\cN$ be an $\cA_X$-submodule of an $\cA_X^{loc}$-module $\cM$. We say that $\cN$ generates $\cM$ if $\mu^{loc}$ is an isomorphism.

\item A coherent $\cA_X$-submodule of a coherent $\cA_X^{loc}$-module $\cM$ is called an $\cA_X$-lattice of $\cM$ if $\cN$ generates $\cM$.

\end{enumerate}
\end{defi}

We recall the definition of good module and refer the reader to \cite[p. 74]{KS3} for a thorough treatment of this notion.
\begin{defi}
(i) A coherent $\cA_X^{loc}$-module is good if, for any relatively compact open subset $U$ of $X$, there exists an $\cA_X |_U$-lattice of $\cM|_U$. 

\noindent (ii) One denotes by $\Mod_{\gd}(\cA_X^{loc})$ the full subcategory of $\Mod_{\coh}(\cA_X^{loc})$ the objects of which are the good modules. 

\noindent (iii) One denotes by $\Der^b_{\gd}(\cA_X^{loc})$ the full subcategory of $\Der^b_{\coh}(\cA_X^{loc})$ the objects of which have good cohomology. 
\end{defi}

The Theorem \ref{thm:dualite} extends to $\cA_X^{loc}$-modules by replacing the category $\Der^b_{\coh}(\cA_X)$ with $\Der^b_{\gd}(\cA^{loc}_X)$. Hence, we have the following statement for $\cA_X^{loc}$-modules.

\begin{thm}\label{thm:dualiteloc}
Let $(X_i, \cA_i)$ $(i=1,2,3)$ be a complex manifold endowed with a DQ-algebroid $\cA_{i}$. Let $Z_i$ be a closed subset of $X_i \times X_{i+1}$ and assume that $Z_1 \times_{X_2} Z_2$ is proper over $X_1 \times X_3$. Set $Z=p_{13} (p_{12}^{-1} Z_1 \cap p_{23}^{-1} Z_2)$. Let $\cK_i \in \Der_{\gd, Z_i}^b(\cA^{loc}_{i{(i+1)}^a})$  $(i=1,2)$. Then the object $\cK_1 \underset{X_2}{\circ} \cK_2$ belongs to $\Der_{\gd, Z}^b(\cA^{loc}_{13^a})$  and we have a natural isomorphism
\begin{equation*}
\Du_{\cA^{loc}_{1{2^a}}}^\prime(\cK_1) \underset{X_{2^a}}{\circ} \omega^{loc}_{X_{2^a}} \underset{X_{2^a}}{\circ} \Du_{\cA^{loc}_{2{3^a}}}^\prime(\cK_2) \stackrel{\sim}{\to} \Du_{\cA^{loc}_{1{3^a}}}^\prime(\cK_1 \underset{X_2}{\circ} \cK_2).
\end{equation*}
\end{thm}

\subsection{Holonomic DQ-modules} 
In this subsection we will be concerned with DQ-modules on complex symplectic manifolds. Thus, we first recall some notions of symplectic geometry. 

\begin{defi} Let $(X,\omega)$ be a complex symplectic manifold.
\begin{enumerate}[(i)]
\item A locally closed complex analytic subvariety $\Lambda$ of X is isotropic if there exists a dense open complex manifold $\Lambda^\prime \subset \Lambda$ such that $\w|_{\Lambda^\prime}=0$.
\item A subvariety $\Lambda$ in $X$ is co-isotropic if for every function $f$ and $g$ such that $f|_\Lambda=g|_\Lambda=0$ then $\lbrace f, g \rbrace=0$ on $\Lambda$ where $\lbrace \cdot, \cdot \rbrace$ is the Poisson bracket associated to the symplectic form $\omega$.
\item A locally closed complex analytic subvariety $\Lambda$ of X is Lagrangian if it is isotropic and co-isotropic.
\end{enumerate}
\end{defi}

Following \cite{KS3}, we review the notion of holonomic DQ-modules. In this section, $X$ is a complex manifold of dimension $d_X$ endowed with a DQ-algebroid $\cA_X$ such that the associated Poisson structure is symplectic and we set $n=\frac{d_X}{2}$.

\begin{defi}
\begin{enumerate}[(i)]
\item An $\cA_X^{loc}$-module $\cM$ is holonomic if it is coherent and if its support is a Lagrangian subvariety of $X$.

\item An $\cA_X$-module $\cN$ is holonomic if it is coherent, without $\hbar$-torsion and $\cN^{loc}$ is a holonomic $\cA_X^{loc}$-module.
\end{enumerate}
\end{defi}

Recall the following lemma.

\begin{lemme}[{\cite[Lemma 2.1]{KR}}] \label{lem:annulation}
Let $r$ be an integer and $\cM$ be a coherent $\cA_X^{loc}$-module so that $\fExt^j_{\cA_X^{loc}}(\cM,\cA_{X}^{loc})=0$ for any $j >r$. Then $\Hn^j_Y(\cM)=0$ for any closed analytic subset $Y$ of $X$ and any $j < \codim Y -r$.
\end{lemme}

Holonomic DQ-modules enjoy the following properties.

\begin{lemme}[{\cite[Lemma 7.2.2]{KS3}}] \label{lem:holref}
Let $\cM$ be a holonomic $\cA_X^{loc}$-module. Then $\Du^\prime_{\cA_X^{loc}}(\cM)[d_X/2]$ is concentrated in degree zero and is holonomic.
\end{lemme}

\begin{lemme}\label{lem:vanishing}
Let $\cM$ be a holonomic $\cA_X^{loc}$-module supported in a Lagrangian subvariety $\Lambda$ of $X$ and $Y$ a closed analytic subset of $\Lambda$. Then $\Hn^j_Y(\cM)=0$ for $0 \leq j < \codim_\Lambda Y$. 
\end{lemme}

\begin{proof}
Recall that $n=\frac{d_X}{2}$.
It follows from Lemma \ref{lem:holref} that
\begin{equation*}
\fExt^j_{\cA_X^{loc}}(\cM,\cA_{X}^{loc})=0 \; \textnormal{for any} \; j > n. 
\end{equation*}
Then by Lemma \ref{lem:annulation}, $\Hn^j_Y(\cM)=0$ for $0 \leq j < \codim_\Lambda Y$.
\end{proof}

\begin{lemme} \label{lem:exttor}
Let $n=\dfrac{d_X}{2}$ and $\cN$ be a holonomic $\cA_X$-module. Then $\fExt^j_{\cA_X}(\cN,\cA_X)=0$ for $j < n$ and $\fExt^n_{\cA_X}(\cN,\cA_X)$ has no $\hbar$-torsion.
\end{lemme}

\begin{proof}
Since $\cN$ has no $\hbar$-torsion, we have $\Supp(\cN^{loc})=\Supp(\cN)$. Thus, the analytic set $\Supp(\gr_\hbar \cN)=\Supp(\cN^{loc})$ is Lagrangian. It follows that 
\begin{equation*}
\codim \Supp(\gr_\hbar \cN) = n. 
\end{equation*}
This implies that 
\begin{equation*}
\Hn^{j}(\gr_\hbar \fRHom_{\cA_X}(\cN, \cA_X))\simeq\fExt_{\cO_X}^{j}(\gr_\hbar \cN, \cO_X)=0 \; \mathrm{for} \; j<n.
\end{equation*}
Thus, by Proposition 1.4.5 of \cite{KS3}, $\fExt_{\cA_X}^j(\cN,\cA_X)=0$ for $j<n$ and $\fExt_{\cA_X}^n(\cN,\cA_X)$ has no $\hbar$-torsion.
\end{proof}

\begin{cor}\label{cor:stillhol}
Let $\cN$ be an holonomic $\cA_X$-module. Then the module $\fExt_{\cA_X}(\cN,\cA_X)$ is holonomic.
\end{cor}

\begin{proof}
It is immediate that the module $\fExt_{\cA_X}(\cN,\cA_X)$ is coherent. By Lemma \ref{lem:exttor}, it has no $\hbar$-torsion. Moreover, $\lbrack \fExt_{\cA_X}(\cN,\cA_X) \rbrack^{loc}$ is holonomic.
\end{proof}

\begin{prop}\label{prop:hollattice}
Let $\cM$ be a holonomic $\cA_X^{loc}$-module such that it has a lattice $\cN$. Then $\fExt_{\cA_X}^n(\fExt_{\cA_X}^n(\cN,\cA_X),\cA_X)$ is a lattice of $\cM$.
\end{prop}

\begin{proof}

We set $\cF=\fExt_{\cA_X}^n(\cN,\cA_X)$. By functoriality, there is a morphism 
\begin{equation*}
\fRHom_{\cA_X}(\cF,\cA_X) \to \fRHom_{\cA_X^{loc}}(\cA_X^{loc} \te_{\cA_X} \cF,\cA_X^{loc})
\end{equation*} 
and
\begin{equation*}
\cA_X^{loc} \te_{\cA_X} \cF = \cA_X^{loc} \te_{\cA_X} \fExt_{\cA_X}^n(\cN,\cA_X) \simeq \fExt_{\cA_X^{loc}}^n(\cM,\cA_X^{loc}).
\end{equation*}
Thus, we have a morphism
\begin{equation*}
\fRHom_{\cA_X}(\cF,\cA_X) \to \fRHom_{\cA_X^{loc}}(\fExt_{\cA_X^{loc}}^n(\cM,\cA_X^{loc}),\cA_X^{loc}).
\end{equation*}

Since $\cM$ is a holonomic module 
\begin{equation*}
\fRHom_{\cA_X^{loc}}(\fExt_{\cA_X^{loc}}^n(\cM,\cA_X^{loc}),\cA_X^{loc}) \simeq \cM [-n].
\end{equation*}

Taking the degree $n$ cohomology, we obtain a morphism

\begin{equation*}
\gamma:\fExt^n_{\cA_X}(\fExt_{\cA_X}^n(\cN,\cA_X),\cA_X) \to \cM
\end{equation*}
such that
\begin{equation*}
\gamma^{loc}: \lbrack \fExt^n_{\cA_X}(\fExt_{\cA_X}^n(\cN,\cA_X),\cA_X) \rbrack^{loc} \to \cM^{loc} \simeq \cM
\end{equation*}
is an isomorphism.

By Lemma \ref{lem:exttor} and Corollary \ref{cor:stillhol}, $\fExt^n_{\cA_X}(\fExt_{\cA_X}^n(\cN,\cA_X),\cA_X)$ has no $\hbar$-torsion. Thus, $\Ker \gamma$ has no $\hbar$-torsion. Hence, there is a monomorphism $\Ker \gamma \hookrightarrow (\Ker \gamma)^{loc}$. Since, $\cA_X^{loc} \te_{\cA_X}$ is an exact functor, we have $(\Ker \gamma)^{loc} \simeq \Ker(\gamma^{loc})$ and $\Ker(\gamma^{loc})=(0)$. It follows that $\Ker \gamma=(0)$.
\end{proof}
\section{A coherence criterion}
The aim of this section is to establish Proposition \ref{prop:finitness} which is a coherence criterion for DQ-modules.
Recall that if $X$ is a topological space and $\cF$ is a sheaf of Abelian group on $X$, the support of $\cF$ denoted $\Supp(\cF)$ is the complementary of the union of open sets $U \subset X$ such that $\cF|_U=0$. Following \cite{KV}, we define the notion of $p$-finite sheaf.

\begin{defi} Let $X$ and $Y$ be two Hausdorff spaces and let $p \colon X \to Y$ be a continuous map. Let $\cF$ be a sheaf of Abelian groups. The sheaf $\cF$ is $p$-finite if the restriction of $p$ to the support of $\cF$ is a finite morphism. 
\end{defi}

\subsection{A coherence criterion: the case of $\cO_X$-modules}

We prove an equivalence between certain categories of $\cO_X$-modules and also a coherence criterion for $\cO_X$-modules. We will use these results to deduce their analogue for DQ-modules.
The statements of this subsection should be compared with \cite{Houzel}.

In all this subsection, we let $D$ be an open subset of some $\C^p$, we set $X= D \times \C^n$ and 
\begin{equation*}
p: D \times \C^n \to D 
\end{equation*}
the projection $D \times \C^n \ni (y,z)  \mapsto y$. 

We write  $\Mod_{\coh}(\cO_{X})$ for the category of coherent $\cO_X$-modules on $X$ and we denote by $\Mod_{\coh}^{p-\fin}(\cO_{X})$ the full sub-category of  $\Mod_{\coh}(\cO_{X})$ the objects of which are the $p$-finite modules.

We set $\cO_D[t]:=\cO_D[t_1,\ldots,t_n]$ and recall that the sheaf $\cO_D[t]$ is a Noetherian sheaf of ring (cf. \cite[Theorem A.31]{KDmod}).

The projection $p:X \to D$ clearly induces a morphism of ringed spaces 
\begin{equation*}
(p,p^\dagger):(X,\cO_X) \to (D,\cO_D[t]). 
\end{equation*}
We consider the following functors
\begin{align*}
\begin{aligned}
p_\ast:&  \Mod(\cO_X) \to  \Mod(\cO_D[t]) \\
       & \hspace{1.3cm} \cF \mapsto  p_\ast \cF
\end{aligned}
&&
\begin{aligned}
p^\sharp:& \Mod(\cO_D[t]) \to \Mod(\cO_X) \\
       & \hspace{1.3cm} \cG \mapsto \cO_X \te_{p^{-1}\cO_D[t]} p^{-1} \cG.
\end{aligned}
\end{align*}

\begin{prop}
The functors $p^\sharp$ is left adjoint to $p_\ast$.
\end{prop}

\begin{proof}
Clear.
\end{proof}

\begin{prop}
The functor $p^\sharp$ is exact.
\end{prop}

\begin{proof}
This follows from the fact that $\cO_X$ is flat over $p^{-1}\cO_D[t]$  (cf. \cite[Example 4.11]{Mal}). 
\end{proof}

\begin{Notation}
We denote by $\Mod_{\cO_D-\coh}(\cO_D[t])$ the full subcategory of the category $\Mod_{\coh}(\cO_D[t])$ the objects of which are also $\cO_D$-coherent modules.
\end{Notation}

\begin{prop}\label{prop:past}
If $\cF \in \Mod_{\coh}^{p-\fin}(\cO_X)$, then $p_\ast \cF \in \Mod_{\cO_D-\coh}(\cO_D[t])$.
\end{prop}

\begin{proof}
The sheaf $p_\ast\cF$ is clearly an $\cO_D[t]$-module and it is $\cO_D$-coherent since $p|_{\Supp(\cF)}$ is finite.
\end{proof}

\begin{prop}\label{prop:psharp}
 If $\cG \in \Mod_{\cO_D-\coh}(\cO_D[t])$, then $p^\sharp \cG \in \Mod_{\coh}^{p-\fin}(\cO_X)$.
\end{prop}

\begin{proof}
Since $\cG$ is a coherent $\cO_D$-module, it is a coherent $\cO_D[t]$-module. Since $ \cO_X$ and $\cO_D[t]$ are coherent sheaves of rings, it  implies that $p^\sharp \cG$ is a coherent $\cO_X$-module.  

Let $y \in D$. Consider $\Ann_{\cO_D[t]}(\cG)$ the annihilator of $\cG$ in $\cO_D[t]$. Since $\cG$ is coherent, we have
\begin{equation*}
\Ann_{\cO_D[t]}(\cG)_y \simeq \Ann_{\cO_{D,y}[t]}(\cG_y).
\end{equation*}

The $t_i, \; 1 \leq i \leq n$ define $\cO_{D,y}$-linear endomorphisms $T_i:\cG_y \to \cG_y, m \mapsto t_i \cdot m$, $1\leq i \leq n$ and $\cG_y$ is a finitely generated module over $\cO_{D,y}$.  Thus, their exist monic polynomials $\w_{i,y} \in \cO_{D,y}[t_i]$ such that $\w_{i,y}(T_i)=0$. This implies that $\w_{i,y} \in  \Ann_{\cO_{D,y}[t]}(\cG_y)$. It is well known that 
\begin{equation*}
\Supp(p^\sharp \cG)=\lbrace x \in X | \Ann_{\cO_X}(p^\sharp\cG)_x \subset \mathfrak{m}_x \rbrace
\end{equation*}
where $\mathfrak{m}_x$ is the maximal ideal of $\cO_{X,x}$. Since $p^\sharp$ is exact and $\cG$ is coherent, 
\begin{equation*} 
\Ann_{\cO_X}(p^\sharp\cG) \simeq p^\sharp \Ann_{\cO_D[t]}(\cG).
\end{equation*}
It follows that 
\begin{equation*}
\Supp(p^\sharp \cG)=\lbrace x \in X | \cO_{X,x} \te_{\cO_{D,y}[t]} \Ann_{\cO_{D,y}[t]}(\cG_y) \subset \mathfrak{m}_x \rbrace.
\end{equation*}

This means that if an element $x_0=(y_0,z_0)$ of $X$ belongs to $\Supp(p^\sharp \cG)$ then, for every $1 \leq i \leq n, \; \w_{i,y_0}(z)=0$. The $\w_{i,y_0}(t_i)$ being monic polynomials, they have a finite number of roots. This implies that $p^{-1}(y)$ is finite.

Let us prove that $p|_ {\Supp(p^\sharp \cG)}$ is closed.  Since being closed is a property local on the base, it is sufficient to check that $p|_{\Lambda } :\Lambda  \to U_i$ is closed for a covering $(U_i)_{i \in I}$ of $D$.

Let $y_0 \in D$ and consider the polynomials $\w_{i,y_0}(t_i) \in \cO_{D,y_0}[t_i], \; 1\leq i \leq n$ previously defined. There exists an open neighbourhood $U$ of $y_0$ such that the $w_i(t_i) \in \cO_{D}[t_i]|_U$ are well defined and belong to $\Ann_{\cO_{D}[t]}(\cG)|_U$. It follows immediately that they induce Weierstrass polynomials $\w_i(y,z_i)  \in \Ann_{\cO_X}(p^\sharp \cG)|_{U \times \C^n}$. 

We set $Z= \lbrace (y,z) \in U \times \C^n | \w_1(y,z_1)=\ldots=\w_n(y,z_n)=0\rbrace$ and $\Lambda=\Supp(p^\sharp \cG)\cap U \times \C^n$. We have the following commutative diagram
\begin{equation*}
\xymatrix{\Lambda \ar[d]_-{i_\Lambda}\ar[r]^-{p|_{\Lambda}} & U\\
 Z \ar[ru]_-{p|_Z} & \\
}
\end{equation*}

The map $i_\Lambda$ is a closed immersion. Thus, we are reduced to show that $p|_Z$ is a closed map. This follow immediately by repeated applications of the Continuity of the Roots theorem (see \cite[p.52]{GR}).
\end{proof}

\begin{lemme}\label{lem:vannak}
Let $(T,\cO_T)$ be a closed analytic subspace of $(X,\cO_X)$ such that $\cO_T$ is $p$-finite. 
\begin{enumerate}[(i)]
\item If $x \in X$ and $y=p(x)$, then for every $x_i \in p^{-1}(y)$, such that $x_i \neq x$, $\cO_{X,x} \te_{\cO_{D,y}[t]} \cO_{T,x_i}=0$. 

\item $p^\sharp p_\ast \cO_T \simeq \cO_T$. 
\end{enumerate}
\end{lemme}

\begin{proof}
\noindent (i) It follows from Propositions \ref{prop:past} and \ref{prop:psharp} that $p^\sharp p_\ast \cO_T$ is a $p$-finite coherent $\cO_X$-module. Thus,
\begin{equation*}
(p^\sharp p_\ast \cO_T)_x \simeq \bigoplus_{x_i \in  p^{-1}(y)} \cO_{X,x} \te_{\cO_{D,y}[t]} \cO_{T,x_i}
\end{equation*}
is finitely generated $\cO_{X,x}$-module. Moreover, $\cO_{X,x}$ is a Noetherian ring. This implies that the module $\cO_{X,x} \te_{\cO_{D,y}[t]} \cO_{T,x_i}$ is also a finitely generated $\cO_{X,x}$-module. Since $x_i \neq x$, there exists a polynomial $P \in \cO_{D,y}[t]$ such that $P(x)=0$ and $P(x_i)\neq 0$. Denoting by $(P)$ the ideal generated by $P$ in $\cO_{X,x}$, we have that $(P)(\cO_{X,x} \te_{\cO_{D,y}[t]} \cO_{T,x_i})=\cO_{X,x} \te_{\cO_{D,y}[t]} \cO_{T,x_i}$ and $(P)$ is included in the maximal ideal of $\cO_{X,x}$. Hence, Nakayama's lemma implies that $\cO_{X,x}\te_{\cO_{D,y}[t]} \cO_{T,x_i}=0$.

\noindent (ii) It is a direct consequence of (i).
\end{proof}
Theorem \ref{thm:equicom} below is well-known to experts and should be compared with \cite[Section 3]{Houzel}. We give a proof for the convenience of the reader. 
\begin{thm}\label{thm:equicom}
The functors
\begin{align*}
\xymatrix{
 \Mod_{\coh}^{p-\fin}(\cO_X) \ar@<.5ex>[r]^-{p_\ast}&  \Mod_{\cO_D-\coh}(\cO_D[t])  \ar@<.5ex>[l]^-{p^\sharp}\\
}
\end{align*}
are equivalences of categories inverse to each other.  
\end{thm}

\begin{proof} 

Consider the morphism
\begin{equation}\label{mor:prem}
p^\sharp p_\ast \cF \to \cF.
\end{equation}
We prove that it is an isomorphism. We set $T=\Supp(\cF)$ and we define $\cO_T=\cO_X / \Ann_{\cO_X}(\cF)$. Let $x \in X$ and set $y=p(x)$. Since $p|_T$ is finite, we have
\begin{align*}
(p^\sharp p_\ast \cF)_x &\simeq  \cO_{X,x} \te_{\cO_{D,y}[t]} (p_\ast \cF)_x\\
                                      &\simeq \bigoplus_{x_i \in p^{-1}(y)}  \cO_{X,x} \te_{\cO_{D,y}[t]} \cF_{x_i}\\
                                      &\simeq \bigoplus_{x_i \in p^{-1}(y)}  \cO_{X,x} \te_{\cO_{D,y}[t]} \cO_{T,x_i} \te_{\cO_{T,x_i}} \cF_{x_i}.
\end{align*}
By lemma \ref{lem:vannak}, $\cO_{X,x} \te_{\cO_{D,y}[t]} \cO_{T,x_i}=0$ if $x \neq x_i$. It follows that $(p^\sharp p_\ast \cF)_x \simeq \cF_x$.  Then the morphism \eqref{mor:prem} is an isomorphism.

Consider the map
\begin{equation}\label{map:equi}
\cG \to p_\ast p^\sharp \cG  .
\end{equation}
We prove it is an isomorphism.

Let $y \in D$. Then the map \eqref{map:equi}  induces a morphism
\begin{equation}\label{mor:isostalk}
\cG_y \to (p_\ast p^\sharp \cG)_y  .
\end{equation}
We construct the inverse of the morphism \eqref{mor:isostalk}. Since $p|_{\Supp(p^\sharp \cG)}$ is finite, we have
\begin{align*}
(p_\ast p^\sharp \cG)_y \simeq \bigoplus_{x_i \in p^{-1}(y) } \cO_{X, x_i} \te_{\cO_{D,y}[t]} \cG_y\simeq (\bigoplus_{x_i \in p^{-1}(y) } \cO_{X, x_i} )\te_{\cO_{D,y}[t]} \cG_y.
\end{align*}

As already explained, the $t_i, \; 1 \leq i \leq n$ define $\cO_{D,y}$-linear endomorphisms $T_i:\cG_y \to \cG_y, v \mapsto t_i \cdot v, \; 1 \leq i \leq n$ and $\cG_y$ is a finitely generated module over $\cO_{D,y}$.  Thus, there exists monic polynomials $\w_{i,y} \in \cO_{D,y}[t_i]$ such that $\w_{i,y}(T_i)=0$. Shrinking $D$ if necessary (and  $X$ since $X=D \times \C^n$), we set $\mathcal{Q}=(\cO_X / (\w_1, \ldots, \w_n) \cO_X)$ and denote by $m$ the cardinal of $p^{-1}(y)$.
By the generalized Weierstrass division theorem (\cite[Ch. 2 §4]{GR}) for any choice of $m$ germs$ f_j \in \mathcal{Q}_{x_j} , \; x_j \in p^{-1}(y), \; i \leq j \leq m$, there exists a uniquely determined polynomial $r \in \cO_{D,y}[t]$ whose degree in $t_i$ is less than $\deg \w_i$, $1 \leq i \leq n$ and such that the residue epimorphism $ \cO_{X,x_j} \to \mathcal{Q}_{x_j}$ maps $r$ onto $f_j$ for all $j=1,\ldots,m$. We define the morphism
\begin{align*}
\phi: (\bigoplus_{x_i \in p^{-1}(y) } \cO_{X, x_i}) \te_{\cO_{D,y}[t]} \cG_y & \to \cG_y\\
(f_{x_i})_{x_i \in p^{-1}(y) }\te n &\mapsto  rn
\end{align*}
where $r$ is the polynomial obtained via the generalized Weierstrass division Theorem. It is immediate that $\phi$ is the inverse of the morphism \eqref{mor:isostalk}. This implies that morphism \eqref{map:equi} is an isomorphism.
\end{proof}

We now prove the following coherence criterion.
\begin{prop} \label{prop:cohcrit} Let $\cF$ be a $\cO_{X}$-module and assume that $\cF$ is a $p$-finite $\cO_X$-module. Then $\cF$ is a coherent $\cO_{X}$-module if and only if $p_\ast \cF$ is a coherent $\cO_{D}$-module.
\end{prop}

\begin{proof}
\noindent (i) Assume that $\cF$ is coherent. Since $p|_{\Supp(\cF)}$ is finite, $p_\ast \cF$ is coherent.\\

\noindent (ii) Assume that $p_\ast \cF$ is coherent. The functor $p^\sharp$ and $p_\ast$ are adjoint. We denote by $\varepsilon: p^\sharp p_\ast \to \id$ the co-unit of the adjunction and by $\eta: \id \to p_\ast p^\sharp $ the unit of the adjunction. Recall that the natural transformation
\begin{equation*}
p_\ast \stackrel{\eta p_\ast }{\to} p_\ast p^\sharp p_\ast \stackrel{p_\ast \varepsilon }{\to} p^\ast
\end{equation*}
is the identity. Thus $p_\ast \varepsilon $ is an epimorphism of sheaves.

Let $\cQ$ be the cokernel of $\varepsilon_\cF :p^\sharp p_\ast \cF \to \cF$. By Propositions \ref{prop:past} and \ref{prop:psharp}, $p^\sharp p_\ast \cF$ is a $p$-finite coherent $\cO_X$-module and $\cQ$ is also a $p$-finite sheaf since $\Supp(\cQ) \subset \Supp(\cF)$. It follows that we have a right exact sequence of $p$-finite sheaves
\begin{equation}\label{eq:precursor}
p^\sharp p_\ast \cF \stackrel{\varepsilon_\cF}{\to} \cF \to \cQ \to 0.
\end{equation}
Since all the sheaves in the exact sequence \eqref{eq:precursor} are $p$-finite, we obtain by applying the functor $p_\ast$ the right exact sequence
\begin{equation*}
p_\ast p^\sharp p_\ast \cF \stackrel{p_\ast \varepsilon_\cF}{\to} p_\ast \cF \to p_\ast \cQ \to 0.
\end{equation*}
Since $p_\ast \varepsilon_\cF$ is an epimorphism of sheaves, it follows that $p_\ast \cQ \simeq 0$ which implies that $\cQ \simeq 0$. 
Thus, $\varepsilon_\cF$ is an epimorphism. It results that $\cF$ is a locally finitely generated $\cO_X$-module and that $\Ann_{\cO_X}(p^\sharp p_\ast\cF) \subset \Ann_{\cO_X}(\cF)$. 
 Setting $\cO_T= \cO_X / \Ann_{\cO_X}(p^\sharp p_\ast\cF)$, it follows from the inclusion $\Ann_{\cO_X}(p^\sharp p_\ast\cF) \subset \Ann_{\cO_X}(\cF)$ that $\cF$ is a $\cO_T$-module. Moreover,
\begin{equation*}
\cO_{X,x} \te_{\cO_{D,y}[t]} \cF_{x_i}\simeq \cO_{X,x} \te_{\cO_{D,y}[t]} \cO_{T,x_i} \te_{\cO_{T,x_i}} \cF_{x_i},
\end{equation*} 
and by Lemma \ref{lem:vannak},
\begin{equation*}
\cO_{X,x} \te_{\cO_{D,y}[t]} \cO_{T,x_i}=0 \; \textnormal{if} \; x_i \neq x.
\end{equation*}
Then we have the following commutative diagram.
\begin{equation*}
\xymatrix{(p^\sharp p_\ast \cF)_x \ar[d]^-{\wr} \ar[r]^-{\varepsilon_{\cF,x}}& \cF_x \ar[d]^-{\wr}\\
\underset{x_i \in p^{-1}(y)}{\bigoplus}  \cO_{X,x} \te_{\cO_{D,y}[t]} \cF_{x_i} \ar[r]  \ar[d]^-{\wr} & \cF_x \ar[d]^-{\wr}\\
\cF_x \ar[r]^-{\sim} &\cF_x.}
\end{equation*}
Hence, $p^\sharp p_\ast \cF \simeq \cF$ which implies that $\cF$ is coherent. 
\end{proof}

\subsection{A coherence criterion: the case of DQ-modules}\label{subsec:cohDQ}
We now focus our attention on DQ-modules and prove an analogue for DQ-modules of Theorem \ref{thm:equicom} and Proposition \ref{prop:cohcrit}. 

Let $M$ be an open subset of $\C^n$. We set $X=T^\ast M$ and denote by $\rho:T^\ast M \to M$ the projection on the base. We denote by $\cA_{X}$ the Moyal-Weyl star-algebra on X and define the algebra 
\begin{align*}
\cB_{M} = & \cO_{M}^\hbar \te_\C \C[\hbar \partial_1,\cdots,\hbar \partial_n].
\end{align*}
Then
\begin{align*}
\cB_{M}^{loc} \simeq & \cO_{M}^{\hbar,loc} \te_\C \C[\partial_1,\cdots,\partial_n].
\end{align*}
\begin{defi}
We say that a $\cB_{M}$-submodule $\mathcal{L}$ of a coherent $\cB_{M}^{loc}$-module $\cN$ is a finiteness $\cB_{M}$-lattice of $\cN$ if $\mathcal{L}$ is a $\cB_{M}$-lattice of $\cN$ and it is coherent as a $\cO_{M}^\hbar$-module.  
\end{defi}

 We denote by $\Mod_{\fgd}(\cB_{M}^{loc})$ the Abelian full-subcategory of $\Mod(\cB_{M}^{loc})$  the objects of which are the $\cB_{M}^{loc}$-modules admitting locally a finiteness lattice.

Let $(x,u)$ be a symplectic coordinate system on $X=T^\ast M$. Then
\begin{align*}
\rho&: X \to M\\
    &  (x,u) \mapsto x.
\end{align*}
We denote by $\Mod_{\coh}^{\rho-\fin}(\cA_{X})$ (resp. $\Mod_{\coh}^{\rho-\fin}(\cA_{X}^{loc})$) the full sub-category of the category $\Mod_{\coh}(\cA_{X})$ (resp. $\Mod_{\coh}(\cA_{X}^{loc})$) the objects of which are the $\rho$-finite modules.

\begin{Rem}\label{lem:finhol}
A coherent $\rho$-finite $\cA_{X}^{loc}$-module $\cM$ is holonomic.
\end{Rem}
 
We have a morphism of $\C^\hbar$-algebras 
\begin{equation*}
\rho^\dagger:\rho^{-1}\cO_M^\hbar \to  \cA_X, \; \rho^\dagger(x_i)=x_i \textnormal{ for } 1 \leq i \leq n. 
\end{equation*}
There is also a monomorphism of $\C^\hbar$-algebras from $\rho^{-1}\cB_M$ to $\cA_X$ defined as follows.
\begin{eqnarray}
\psi: \rho^{-1}\cB_{M}  \to \cA_{X}, & \;  x_i  \mapsto  x_i, &  \hbar \partial_i  \mapsto  u_i.
\end{eqnarray} 
We set
\begin{align*}
\begin{aligned}
\rho^{\flat}:&\Mod(\cB_{M})  \to  \Mod(\cA_{X})\\
                 & \cN \mapsto \cA_{X} \te_{\rho^{ -1}\cB_{M}} \rho^{-1} \cN
\end{aligned}
&&
\begin{aligned}
\rho^{\natural}:&\Mod(\cB_{M}^{loc})  \to  \Mod(\cA^{loc}_{X})\\
                 & \cN \mapsto \cA_{X}^{loc} \te_{\rho^{ -1}\cB_{M}^{loc}} \rho^{-1} \cN.
\end{aligned}
\end{align*}

\begin{prop}
 The functors $\rho^\flat$ and $\rho_\ast$ (resp.  $\rho^\natural$ and $\rho_\ast$) form the adjoint pair $(\rho^\flat,\rho_\ast)$  (resp. $(\rho^\natural,\rho_\ast)$).
\end{prop}

\begin{proof}
Clear.
\end{proof}

\begin{prop}
The functor $\rho^\flat$ (resp. $\rho^\natural$) is exact.
\end{prop}

\begin{proof}
We need to prove that for every $\cB_M$-module $\cN$
\begin{equation*}
\Hn^i(\cA_{X} \Lte_{\rho^{-1}\cB_{M}}  \rho^{-1}\cN)=0 \quad for \quad i<0.
\end{equation*}
which is equivalent to prove that for every $z \in X$ and every $\cN \in \Mod(\cB_M)$  
\begin{equation*}
\Hn^i(\cA_{X,z} \Lte_{\rho^{-1}\cB_{M,z}} \rho^{-1} \cN_z)=0 \quad for \quad any \quad i <0.
\end{equation*}

For that purpose, we adapt the proof of Theorem 1.6.6 of \cite{KS3}. Let $z \in X$. We set $A=\cA_{X,z}$, $B=\cB_{M,\rho(z)}$, $N=\cN_{\rho(z)}$, $Q=\cO_{X,z}$ and $R=\cO_{M,z}[t_1,\ldots,t_n]$ and we recall that $Q$ is flat over $R$ (see \cite[example 4.11]{Mal}). The problem reduces to show that $\Tor^i_B(A,N)=0$ for all $i<0$. Since, $\Tor$ commute with filtrant colimits, we can further assume that $N$ is a $B$-module of finite type. Furthermore, every finitely generated $B$-module is an extension of a $B$-module without $\hbar$-torsion by a module of $\hbar$-torsion. So, we just need to treat the case of modules without $\hbar$-torsion and the case of $\hbar$-torsion modules.

\noindent (i) We assume that $N$ has no $\hbar$-torsion, then
\begin{equation*}
\gr_\hbar(A \Lte_B N) \simeq Q \te_{R}  N/ \hbar N.
\end{equation*}
Moreover, $N$ is a finitely generated $B$-module. Hence, $A \Lte_B N$ belongs to $\Der^b_f(A)$. It follows from Theorem 1.6.1 of \cite{KS3} that $A \Lte_B N$ is cohomologically complete. 
Hence, Proposition 1.5.8 of \cite{KS3} implies that $A \Lte_B N$ is concentrated in degree zero.\\
\noindent (ii) We assume that $N$ is of $\hbar$-torsion. Clearly $N=\bigcup_{k \in \N^\ast} N_k$ where \linebreak $N_k=\lbrace n \in N | \hbar^kn=0 \rbrace$ and
\begin{equation*}
A \Lte_B N_1 \simeq Q \Lte_R N_1 \simeq Q \te_R N_1.
\end{equation*}
Thus, $A \Lte_B N_1$ is concentrated in degree zero. By recursion, we extend this to all the $N_k$. Since $\Tor_B$ commutes with filtrant colimits $\Tor^i_B(A,N)=0$  for $i \neq 0$. It follows that for any finitely generated module $N$, $A \Lte_B N$ is concentrated in degree zero.
 
The result for $\rho^\natural$ follows immediately since $\rho^\natural$ is the composition of $\rho^\flat$ with the exact functor $\C_X^{\hbar,loc} \te_{\C^\hbar_X} \cdot$. 
\end{proof}

\begin{prop}\label{prop:rhoast}
If $\cM \in \Mod_{\coh}^{\rho-\fin}(\cA_X)$, then $\rho_\ast \cM$ is an object of $\Mod_{\cO^\hbar_M-\coh}(\cB_M)$.
\end{prop}

\begin{proof}
 Let $\cM \in \Mod^{\coh}_{\rho-\fin}(\cA_X)$. The sheaf $\rho_\ast \cM$ is clearly a $\cB_M$-module and by Grauert's direct image theorem for DQ-modules (cf. \cite[Thm 3.2.1 ]{KS3}) $\rho_\ast \cM$ is $\cO_M^\hbar$-coherent.
\end{proof}

\begin{prop}\label{prop:rhosharp}
 If $\cN \in \Mod_{\cO^\hbar_M-\coh}(\cB_M)$, then $\rho^\flat \cN \in \Mod_{\coh}^{\rho-\fin}(\cA_X)$.
\end{prop}

\begin{proof}
The coherence of $\rho^\flat \cN$ is clear. Now, assume that $\cN$ has no $\hbar$-torsion. Then
\begin{equation*}
\Supp ( \rho^\flat \cN)= \Supp( \gr_\hbar \rho^\flat \cN)=\Supp(\rho^\sharp (\cN / \hbar \cN)).
\end{equation*}
It follows that $\rho^\flat \cN \in \Mod_{\coh}^{\rho-\fin}(\cA_{X})$.

Let $\cN$ be a module of $\hbar$-torsion with $\torh(\cN)=1$. We have the following isomorphism
\begin{align*}
\rho^\flat \cN & \simeq \cA_X \te_{\rho^{-1}\cB_M} \rho^{-1} \cN\\
                            &\simeq \cO_X \te_{\rho^{-1} \cO_M[t]} \rho^{-1} \cN\\
                           & \simeq \rho^\sharp \cN.
\end{align*}
Thus, $\rho^\flat \cN$ is a $\rho$-finite $\cA_X$-module. 

Let $\cN$ be a module of $\hbar$-torsion with $\torh(\cN)=k$, $ k \geq 2$. Every $\cB_M$-submodule $\cP$ of $\cN$ with $\torh(\cP)=1$ is $\rho$-finite. Assume that $\rho^\flat \cP$ is $\rho$-finite for each $\cB_M$-submodule $\cP$ of $\cN$ which is of uniform $\hbar$-torsion with $\torh(\cP)=k-1$. Let $\cQ$ be a $\cB_M$-submodule of $\cN$ which is of $\hbar$-torsion with $\torh(\cQ)=k$. Then we have the following short exact sequence
\begin{equation*}
0 \to \hbar Q \to Q \to Q / \hbar Q \to 0.
\end{equation*}
Moreover, $\rho^\flat$ is an exact functor, then
\begin{equation*}
0 \to \rho^\flat (\hbar Q) \to \rho^\flat Q \to \rho^\flat(Q / \hbar Q) \to 0.
\end{equation*}
So, $\Supp(\rho^\flat Q) \subset \Supp( \rho^\flat (\hbar Q)) \cup \Supp(\rho^\flat(Q / \hbar Q))$ which implies that $\rho^\flat Q$ is $\rho$-finite. In particular, $\rho^\flat\cN$ is $\rho$-finite.

Assume that $\cN$ is of $\hbar$-torsion. Thus, there exists an open covering $(U_i)_{i  \in I}$ of $M$ such that $\torh(\cN|_{U_i})=k_i$ with $k_i \in \N$. We set $\rho_i:=\rho|_{\rho^{-1}(U_i)}^{U_i}$. It follows from the previous cases that $\rho_i^\flat \cN|_{U_i}$  is $\rho_i$-finite. Since being finite for a morphism is a property local on the base, the sheaf $\rho^\flat\cN$ is $\rho$-finite.

Let $\cN \in \Mod_{\cO^\hbar_M-\coh}(\cB_{M})$. Then $\cN$ is the extension by an object of $\Mod_{\cO^\hbar_M-\coh}(\cB_{M})$ of $\hbar$-torsion of an object of $\Mod_{\cO^\hbar_M-\coh}(\cB_{M})$ without $\hbar$-torsion. It follows from what is preceding that $\rho^\flat \cN$ is $\rho$-finite.
\end{proof}

The following proposition is the analogue for DQ-modules of \cite[Lemma 4.3]{KV}.
\begin{prop}\label{prop:eqdq}
The functors
\begin{equation*}
\xymatrix{
\Mod_{\coh}^{\rho-\fin}(\cA_{X}) \ar@<.4ex>[r]^-{\rho_\ast} & \ar@<.4ex>[l]^-{\rho^{\flat}}\Mod_{\cO_M^{\hbar}-\coh}(\cB_{M})
 }\end{equation*}
 are equivalences of categories and inverse to each other. 
\end{prop}

\begin{proof}

Once again, we will use the fact that a coherent DQ-module is the extension of a coherent DQ-module without $\hbar$-torsion by a DQ-module of $\hbar$-torsion. 
Let $\cM \in \Mod_{\coh}^{\rho-\fin}(\cA_X)$. The adjunction between $\rho_\ast$ and $\rho^\flat$ provides the map
\begin{equation}\label{map:DQnotorfinal}
\rho^\flat \rho_\ast \cM \to \cM.
 \end{equation}

\noindent (i) Assume that $\cM$ has no $\hbar$-torsion. Applying the $\gr_\hbar$ functor to \eqref{map:DQnotorfinal}, we get
\begin{equation}\label{map:DQnotor}
\rho^\sharp \rho_\ast (\cM / \hbar \cM) \to (\cM / \hbar \cM).
 \end{equation}
The module $(\cM / \hbar \cM)$ belongs to $\Mod_{\coh}^{\rho-\fin}(\cO_X)$, thus by the Theorem \ref{thm:equicom}, it follows that the morphism \eqref{map:DQnotor} is an isomorphism.  Since $\cM$ and $\rho^\flat \rho_\ast \cM$ are complete and have no $\hbar$-torsion, it follows that the morphism \eqref{map:DQnotorfinal} is an isomorphism.

\noindent (ii) Assume that $\cM$ is of $\hbar$-torsion with $\torh(\cM)=1$. Then we have the following commutative diagram
\begin{equation*}
\xymatrix{ \rho^\flat \rho_\ast \cM \ar[r] \ar[d]^-{\wr} & \cM \ar[d]^-{\wr} \\
                 \rho^\sharp \rho_\ast \cM \ar[r]^-{\sim}  & \cM 
}
\end{equation*}
which implies that the map \eqref{map:DQnotorfinal} is an isomorphism.\\
\noindent (iii) If $\cM$ is a module of $\hbar$-torsion with $\torh(\cM)=k$, the fact that the morphism \eqref{map:DQnotorfinal} is an isomorphism is proved by recursion on the index of $\hbar$-torsion.\\ 
\noindent (iv) Assume that $\cM$ is of $\hbar$-torsion. Let $z=(x,u)$ a point of $X=T^\ast M$. Consider the stalk of morphism \eqref{map:DQnotorfinal} at $z$
\begin{equation}\label{map:DQmorstalk}
(\rho^\flat \rho_\ast \cM)_z \to \cM_z.
\end{equation}
If $\rho^{-1}(x) \cap \Supp(\cM)= \emptyset$, it is clear that \eqref{map:DQmorstalk} is an isomorphism.
Assume that $\rho^{-1}(x) \cap \Supp(\cM)= \lbrace z_1, \ldots, z_s \rbrace$. Since $\cM$ is coherent we can find open pairwise disjoint neighbourhoods $U_1^\prime, \ldots, U^\prime_s$ of $z_1, \ldots, z_s$ such that $\overline{U^\prime_i} \cap \overline{U^\prime_j}=\emptyset$ for $i \neq j$ and $\cM|_{U^\prime_i}$ is of $\hbar$-torsion with $\torh(\cM|_{U^\prime_i})=k_i$. It follows from Proposition \ref{prop:sheafdecompo}, that there exists an open subset $V$ of $M$ such that $\rho^{-1}(V)$ is an open neighbourhood of $z_1, \ldots, z_s$ and $\torh(\cM|_{T^\ast V})=k$ with $k=\max_{1 \leq i \leq s} \lbrace k_i \rbrace$.

Using the fact that the morphism \eqref{map:DQnotorfinal} is an isomorphism for $\hbar$-torsion modules with index of $\hbar$-torsion equal to $k$, it follows that
\begin{align*}
(\rho^\flat \rho_\ast \cM )|_{\rho^{-1}(V)} \to \cM|_{\rho^{-1}(V)}.
\end{align*}
is an isomorphism which proves the claim.

\noindent(v) Let $\cM$ be a $\rho$-finite $\cA_X$-module. Then $\cM$ is the extension of a $\rho$-finite module without $\hbar$-torsion by a $\rho$-finite module of $\hbar$-torsion. Using the preceding cases the result follows immediately.

The proof that the unit $\id\to \rho_\ast \rho^\flat$ is an isomorphism is simpler and follows a similar path. 
\end{proof}

\begin{lemme}\label{lem:lattice}
Let $\cM \in \Mod_{\coh}^{\rho-\fin}(\cA_{X}^{loc})$ then, $\cM$ is locally good on the base i.e. there exists an open covering $(U_i)_{i \in I}$ of $M$ such that for every $i \in I$, $\cM|_{\rho^{-1}(U_i)}$ is good.
\end{lemme}

\begin{proof}
If $\rho^{-1}(x) \cap \Supp(\cM) = \emptyset$, then by Lemma \ref{lem:zerostalk}, there exists a neighbourhood $W$ of $x$ such that $\cM|_{W \times \C^n}=0$ and $0$ is a lattice of $\cM$ on $W \times \C^n$. Assume that $\rho^{-1}(x) \cap \Supp(\cM) = \lbrace z_1,\ldots, z_s \rbrace$. Since $\cM$ is coherent, we can find open pairwise disjoint neighbourhoods $U_1^\prime, \ldots, U^\prime_s$ of $z_1, \ldots, z_s$ such that $\overline{U^\prime_i} \cap \overline{U^\prime_j}=\emptyset$ and such that $\cM|_{U^\prime_i}$ admits a lattice $\cM_i$ for $1 \leq i \leq s$. It follows from Proposition \ref{prop:sheafdecompo}, that there exists  pairwise disjoint neighbourhoods $U_1, \ldots, U_s$ of $z_1, \ldots, z_s$ such that $U_i \subset U_i^\prime$ and such that if we write $\gamma_i:U_i \hookrightarrow V \times \C^n$ for the inclusion of $U_i$ into $V \times \C^n$ we have the isomorphism
\begin{equation*}
\cM|_{V \times  \C^n} \stackrel{\sim}{\to} \prod_{i=1}^s \gamma_{i\ast}\gamma_i^{-1} (\cM|_{V \times \C^n}).
\end{equation*}
It follows immediately that $\prod_{i=1}^s \gamma_{i\ast}\gamma_i^{-1} \cM_i$ is a lattice of $\cM|_{V \times \C^n}$.
\end{proof}
The following proposition is the analogue for DQ-modules of \cite[Proposition 4.4]{KV}.
\begin{prop}
The functors
\begin{equation*}
\xymatrix{
\Mod_{\coh}^{\rho-\fin}(\cA_{X}^{loc}) \ar@<.4ex>[r]^-{\rho_\ast} & \ar@<.4ex>[l]^-{\rho^{\natural}}\Mod_{\fgd}(\cB_{M}^{loc})
}
\end{equation*}
are equivalences of categories and inverse to each other.
\end{prop}

\begin{proof}
Using Lemma \ref{lem:lattice} the proof of the proposition reduces to Proposition \ref{prop:eqdq}.
\end{proof}

We will need the following coherence criterion (see \cite[Theorem  1.3.6]{KS3}) in order to establish Proposition \ref{prop:finitness}.
\begin{thm}\label{thm:critcohabelienne}
Let $(T,\cO_T)$ be a complex manifold endowed with a DQ-algebra $\cA_T$. Let $\cM$ be a locally finitely generated $\cA_T$-module. Then $\cM$ is coherent if and only if $\hbar^n \cM / \hbar^{n+1} \cM$ is a coherent $\cO_T$-module for any $n \geq 0$. 
\end{thm}
\begin{prop}\label{prop:finitness}

\noindent (i) Let $\cM$ be a $\rho$-finite $\cA_{X}$-module. Then $\cM$ is a coherent $\cA_{X}$-module if and only if $\rho_\ast \cM$ is a coherent $\cO_{M}^\hbar$-module.

\medskip
\noindent (ii) Let $\cM$ be a $\rho$-finite $\cA^{loc}_{X}$-module.

 (a) If $\cM$ is a coherent $\cA^{loc}_{X}$-module, then $\rho_{\ast} \cM$ is a coherent $\cO_{M}^{\hbar,loc}$-module.

 (b) If $\rho_{\ast} \cM$ is a coherent $\cO_{M}^{\hbar,loc}$-module and has a coherent $\cO_{M}^\hbar$-lattice $\cL$ such that $\cL$ is a $\cB_{M}$-module, then $\cM$ is a coherent $\cA_{X}^{loc}$-module.
\end{prop}

\begin{proof}
\noindent (i) The first direction is Proposition \ref{prop:rhoast}. Let us prove the converse.
Let $\cM$ be a $\rho$-finite module and assume that $\rho_\ast \cM$ is coherent. We denote by $\varepsilon:\rho^\flat \rho_\ast \to \id$ the co-unit and by $\eta:\id \to \rho_\ast \rho^\flat$ the unit of the adjunction $(\rho^\flat, \rho_\ast)$. Recall that the natural transformation
\begin{equation*}
\rho_\ast \stackrel{\eta \rho_\ast }{\to} \rho_\ast \rho^\flat \rho_\ast \stackrel{\rho_\ast \varepsilon }{\to} \rho^\ast
\end{equation*}
is the identity. Thus, $\rho_\ast \varepsilon$ is an epimorphism of sheaves. By an argument similar to the one of the proof of Proposition \ref{prop:cohcrit}, we show that $\varepsilon_\cM$ is an epimorphism of sheaves which implies that $\cM$ is locally finitely generated.

The restriction of $\rho_\ast$ to the category $\Mod^{\rho-\fin }(\cA_X)$ is exact and the sheaves $\cM$, $\hbar^n \cM$ and $\hbar^n  \cM / \hbar^{n+1} \cM$ are $\rho$-finite. This implies that
\begin{equation*}
\hbar^n \rho_\ast \cM / \hbar^{n+1} \rho_\ast \cM \simeq  \rho_\ast(\hbar^n \cM) /  \rho_\ast(\hbar^{n+1} \cM)\simeq \rho_\ast(\hbar^n \cM / \hbar^{n+1} \cM ).
\end{equation*}
This means that the sheaf  $\rho_\ast(\hbar^n \cM / \hbar^{n+1} \cM )$ is a coherent $\cO_M^\hbar$-module and hence, a coherent $\cO_M$-module since $\torh(\rho_\ast(\hbar^n \cM / \hbar^{n+1} \cM ))=1$. By Proposition \ref{prop:cohcrit}, $\hbar^n \cM / \hbar^{n+1} \cM$ is a coherent $\cO_X$-module. It follows from Theorem \ref{thm:critcohabelienne} that $\cM$ is a coherent $\cA_X$-module.\\

\noindent (ii) The point (a) is clear. Let us prove (b). Let $\cM$ be a $\rho$-finite module such that $\rho_{\ast} \cM$ is a coherent $\cO_{M}^{\hbar,loc}$-module and has a coherent $\cO_{M}^\hbar$-lattice $\cL$ such that $\cL$ is a $\cB_{M}$-module. This implies that $\rho^\flat \cL \in \Mod_{\coh}^{\rho-\fin}(\cA_X)$. 

By an argument similar to the one of the proof of Proposition \ref{prop:cohcrit}, we obtain that the co-unit
\begin{equation}\label{map:counitloc}
\rho^\natural \rho_\ast \cM \to \cM
\end{equation}
is an epimorphism. Moreover, there is a canonical map
\begin{equation}\label{map:locmap}
\rho^\flat \cL \to \rho^\natural \rho_\ast \cM. 
\end{equation}
Composing the maps \eqref{map:counitloc}  and \eqref{map:locmap}, we get
\begin{equation}\label{map:imfin}
\theta:\rho^\flat \cL \to \rho^\natural \rho_\ast \cM \to \cM.
\end{equation}
Moreover, $\theta^{loc}:(\rho^\flat \cL)^{loc} \to \cM^{loc}\simeq \cM$ is an epimorphism. This implies that $(\im\theta)^{loc} \simeq \cM$. The $\cA_X$-module $\im \theta$ is $\rho$-finite and the morphism \eqref{map:imfin} induces the epimorphism $\rho^\flat \cL \twoheadrightarrow \im \theta$. Thus, $\rho_\ast \im\theta$ is a locally finitely generated $\cO_M^\hbar$-submodule of the coherent $\cO_M^{\hbar,loc}$-module $\rho_\ast \cM$. It follows by \cite[Lemma 2.3.13]{KS3} that  $\rho_\ast \im\theta$ is a coherent $\cO_M^\hbar$-module. Then point (i) implies that $\im \theta$ is a coherent $\cA_X$-module. Hence, $\cM$ is a coherent $\cA_X^{loc}$-module.
\end{proof}

\begin{Rem}
If $U$ is an open subset of $X$, the above criterion can also be used to prove the coherence of a  $\rho|_U$-finite $\cA_X|_U$-module $\cM$. Indeed, if $i: U \hookrightarrow \rho(U) \times \C^n$ denotes the inclusion of $U$ into $ \rho(U) \times \C^n$ then, by Corollary \ref{cor:finitesheaf} $i_\ast \cM$ is a $\rho|_{\rho(U) \times \C^n}$-finite $\cA_{\rho(U) \times \C^n}$-module the coherence of which is equivalent to the coherence of $\cM$.
\end{Rem}

\section{A duality result}\label{sec:duality}

In this section, we prove the analogue for DQ-modules of \cite[Propositions 4.2 \& 4.6 ]{KV}. We let $X$ be an open subset of $T^\ast M$ and shrinking $M$ is necessary we assume that $\rho(X)=M$. We denote by $\cA_X$ the Moyal-Weil star-algebra on X and, on $M$, we consider the trivial star-algebra $\cA_M:=\cO_{M}^\hbar$. We endow $M \times X$ with the star-algebra $\cA_{M \times {X^a}}=\cO^\hbar_{M} \ubtimes \cA_{{X^a}}$. Following page 61 of \cite{KS3} the star-product on $\cA_{M \times {X^a}}$ is given by
\begin{equation*}
f \star g = \sum_{i \geq 0} \hbar^i (m \boxtimes P_i) (f,g)  \quad \textnormal{for every } f, \; g \in \cO_{M \times X} 
\end{equation*}
where the $P_i$ are the bi-differential operators defining the star-product on $\cA_{{X^a}}$ and $m$ is the multiplication of function on $\cO_{M}^\hbar$.

Consider the projection $\rho: X \to M$ and define 
\begin{align*}
\begin{aligned}
\gamma_\rho \colon & X \times X \to  M \times X \\
                     & (x,s) \mapsto (\rho(x);s)
\end{aligned}
&&
\begin{aligned}
\delta_{\rho} \colon & X \to  M \times X\\
                     & x \mapsto (\rho(x);x).
\end{aligned}
\end{align*}
Notice that $\delta_\rho=\gamma_\rho \circ \delta$. We define the morphism
\begin{align*}
\gamma_{\rho}^\dagger:& \gamma_{\rho}^{-1} \cA_{M \times {X^a}} \to \cA_{{X \times X^a}}\\
                            & \hspace{1.5cm} f \mapsto f \circ \gamma_{\rho}. 
\end{align*}
It follows from the formula defining the star-product on $\cA_{M \times {X^a}}$ and $\cA_{X\times X^a}$ that $\gamma_{\rho}^\dagger$ is a morphism of $\C^\hbar$-algebras. 

The $\cA_X \te \cA_{X^a}$-module $\cA_X$ is bi-invertible. Thus, by \cite[Lemma 2.4.1]{KS3} $\delta_\ast \cA_X$ has a structure of $\cA_{X\times X^a}$-modules. By adjunction, $\gamma_\rho^\dagger$ induces a morphism of algebra $\cA_{M \times X^a} \to \gamma_{\rho\ast}\cA_{X \times X^a}$ which, in turn, induces on $\gamma_{\rho \ast} \delta_\ast \cA_X\simeq\delta_{\rho \ast}\cA_X$ a structure of $\cA_{M \times X^a}$-module. Since $\delta_{\rho}$ is finite, $ \delta_{\rho \ast} \cA_{X}$ is a coherent $\cA_{M \times {X^a}}$-module.

We denote by  $\Der^b_{\rho-\fin}(\cA_{X})$ the full subcategory of  $\Der^b(\cA_{X})$ the objects of which have $\rho$-finite cohomology and define the functor
\begin{align*}
\rho^\ast:& \Mod(\cA_{M^a}) \to \Mod(\cA_{X^a})\\
                 & \hspace{0.5cm}\cN \mapsto \rho^{-1}\cN \te_{\rho^{-1}\cA_M} \cA_X
\end{align*}

\begin{lemme}
Let $\cM \in  \Der^b_{\rho-\fin}(\cA_{X})$ and $\cN \in \Der^b(\cA_{{M}^a})$. Then
 \begin{align*}
 \delta_{\rho \ast}\cA_{{X}} \conv[\cA_ {X}] \cM \simeq \dR \rho_\ast \cM,\\
  \cN \conv[\cA_ {M}] \delta_{\rho \ast}\cA_{{X}}\simeq {\dL \rho}^{\ast} \cN.
 \end{align*}
\end{lemme}

\begin{proof}
These formulas are direct consequences of the projection formula.
\end{proof}

\begin{lemme}\label{lem:operation}
Let $\cM \in  \Der^b_{\coh,\rho-\fin}(\cA_{X^a})$ (resp. in $\Der^b_{\gd,\rho-\fin}(\cA^{loc}_{X^a}$)). Then we have the following isomorphisms
\begin{align}
(\w_{M^a} \conv[\cA_ {M^a}] \Du^\prime_{\cA_{M \times X^a}}( \delta_{\rho \ast}\cA_{{X}})) \conv[\cA_ {X^a}] \cM \buildrel\sim\over \to \rho_\ast \cM \label{operation} \quad  \textnormal{ in } \Der^b_{\coh}(\cA_{M^a}),\\
(\w_{M^a}^{loc} \conv[\cA^{loc}_{M^a}] \Du^\prime_{\cA^{loc}_{M \times X^a}}( \delta_{\rho \ast}\cA^{loc}_{{X}^a})) \conv[\cA^{loc}_{X^a}] \cM \buildrel\sim\over \to \rho_\ast \cM \label{operationbis}  \quad  \textnormal{ in } \Der_\gd^b(\cA^{loc}_{M^a}).
\end{align}
\end{lemme}

\begin{proof}
We  prove formula (\ref{operation}). Let $\cN \in \Der^b_{\coh}(\cA_{M^a})$. We have
\begin{align*}
 \RHom_{\cA_{M^a}}(\cN,(\w_{M^a} \conv[\cA_ {M^a}]\Du^\prime_{\cA_{M \times X^a}}( \delta_{\rho \ast}\cA_{{X}})) \conv[\cA_ {X^a}] \cM) \simeq& \\ 
& \hspace{-9cm} \simeq \Rg (M,\fRHom_{\cA_M}(\cN,(\w_{M^a} \conv[\cA_ {M^a}] \Du^\prime_{\cA_{M \times X^a}}( \delta_{\rho \ast}\cA_{{X}})) \conv[\cA_ {X^a}] \cM)).
\end{align*}
and\\
\scalebox{0.9}{\parbox{\linewidth}{
\begin{align*}
\fRHom_{\cA_{M^a}}(\cN,(\w_{M^a} \conv[\cA_ {M^a}] \Du^\prime_{\cA_{M \times X^a}}( \delta_{\rho \ast}\cA_{{X}})) \conv[\cA_ {X^a}] \cM) &\simeq& \\
& \hspace{-8.5cm} \simeq \Du^\prime_{\cA_{M^a}}(\cN) \Lte_{\cA_{M^a}} \lbrack (\w_{M^a} \conv[\cA_ {M^a}] \Du^\prime_{\cA_{M \times X^a}}( \delta_{\rho \ast}\cA_{{X}})) \conv[\cA_ {X^a}] \cM \rbrack\\
&\hspace{-8.5cm} \simeq \Du^\prime_{\cA_{M^a}}(\cN) \Lte_{\cA_{M^a}} \dR p_{M \ast}((\w_{M^a} \conv[\cA_ {M^a}] \Du^\prime_{\cA_{M \times X^a}}( \delta_{\rho \ast}\cA_{{X}})) \Lte_{p_{X^a}^{-1} \cA_{X^a}} p_{X^a}^{-1} \cM)\\
& \hspace{-8.5cm} \simeq  \dR p_{M \ast} (p_M ^{-1} \Du^\prime_{\cA_{M^a}}(\cN) \Lte_{ p_M ^{-1} \cA_{M^a}} (\w_{M^a} \conv[\cA_ {M^a}] \Du^\prime_{\cA_{M \times X^a}}( \delta_{\rho \ast}\cA_{{X}})) \Lte_{p_{X^a}^{-1} \cA_{X^a}} p_{X^a}^{-1} \cM)
\end{align*}
}}\\%
Applying  $\Rg(M; \cdot)$, we get\\
\scalebox{0.9}{\parbox{\linewidth}{
\begin{align*}
\Rg(M;\dR p_{M \ast} (p_M ^{-1} \Du^\prime_{\cA_{M^a}}(\cN) \Lte_{ p_M ^{-1} \cA_{M^a}} (\Du^\prime_{\cA_{M \times X^a}}( \w_{M^a} \conv[\cA_ {M^a}] \delta_{\rho \ast}\cA_{{X}})) \Lte_{p_{X^a}^{-1} \cA_{X^a}} p_{X^a}^{-1} \cM)) \simeq&\\
& \hspace{-14.2cm} \simeq \Rg(M \times X ; p_M ^{-1} \Du^\prime_{\cA_{M^a}}(\cN) \Lte_{ p_M ^{-1} \cA_{M^a}} (\w_{M^a} \conv[\cA_ {M^a}] \Du^\prime_{\cA_{M \times X^a}}( \delta_{\rho \ast}\cA_{{X}})) \Lte_{p_{X^a}^{-1} \cA_{X^a}} p_{X^a}^{-1} \cM)\\
& \hspace{-14.2cm} \simeq \Rg(X ; \dR p_{X \ast}( p_M ^{-1} \Du^\prime_{\cA_{M^a}}(\cN) \Lte_{ p_M ^{-1} \cA_{M^a}} (\w_{M^a} \conv[\cA_ {M^a}] \Du^\prime_{\cA_{M \times X^a}}( \delta_{\rho \ast}\cA_{{X}})) \Lte_{p_{X^a}^{-1} \cA_{X^a}} p_{X^a}^{-1} \cM))\\
& \hspace{-14.2cm} \simeq \Rg(X ; (\Du^\prime_{\cA_{M^a}}(\cN) \Lte_{ p_M ^{-1} \cA_{M^a}} (\w_{M^a} \conv[\cA_ {M^a}] \Du^\prime_{\cA_{M \times X^a}}( \delta_{\rho \ast}\cA_{{X}})) \Lte_{p_{X^a}^{-1} \cA_{X^a}} p_{X^a}^{-1} \cM))\\
& \hspace{-14.2cm} \simeq \Rg(X ; \lbrack \Du^\prime_{\cA_{M^a}}(\cN)) \conv[\cA_{M^a}] \w_{M^a} \conv[\cA_ {M^a}] \Du^\prime_{\cA_{M \times X^a}}( \delta_{\rho \ast}\cA_{{X}}) \rbrack \Lte_{ \cA_{X^a}}  \cM)).
\end{align*}
}}%

\noindent Applying Theorem \ref{thm:dualite}, we obtain
\begin{align*}
\Rg(X ; \lbrack \Du^\prime_{\cA_{M^a}}(\cN)) \conv[\cA_{M^a}] \w_{M^a}  \conv[\cA_ {M^a}] \Du^\prime_{\cA_{M \times X^a}}( \delta_{\rho \ast}\cA_{{X}}) \rbrack \Lte_{ \cA_{X^a}}  \cM)) \simeq &\\
& \hspace{-8cm} \simeq \Rg(X ; \Du^\prime_{\cA_{X^a}}( \cN \conv[\cA_{M^a}] \delta_{\rho \ast}\cA_{{X}}))  \Lte_{ \cA_{X^a}}  \cM)) \\
& \hspace{-8cm} \simeq \RHom_{\cA_{X^a}}(\dL \rho^\ast \cN, \cM)\\
& \hspace{-8cm} \simeq \RHom_{\cA_{M^a}}(\cN, \dR\rho_\ast \cM)
\end{align*}
It follows from the Yoneda Lemma that
\begin{equation*}
(\w_{M^a} \conv[\cA_ {M^a}] \Du^\prime_{\cA_{M \times X^a}}( \delta_{\rho \ast}\cA_{{X}})) \conv[\cA_ {X^a}] \cM \simeq \dR\rho_\ast \cM.
\end{equation*}

The construction of the morphism \eqref{operationbis} is similar. First, one assumes that $\cN \in \Der^b_{\gd}(\cA_X^{loc})$ and instead of using Theorem \ref{thm:dualite}, one uses Theorem \ref{thm:dualiteloc} and get the desired isomorphism.
\end{proof}

Finally, we get the following duality result which is the analogue for DQ-modules of \cite[Proposition 4.2]{KV}.

\begin{prop}\label{prop:dual}
Let $\cM \in \Mod_{\coh}^{\rho-\fin}(\cA_{X})$ (resp. in $\Mod_{\gd}^{\rho-\fin}(\cA_{X}^{loc}))$. Then we have
\begin{align}
\rho_\ast \fExt^k_{\cA_{X}}(\cM, \cA_{X}) \simeq \fExt^{k-n}_{\cO_{M}^{\hbar}}(\rho_\ast\cM, \cO_{M}^{\hbar}), \label{for:pasloc}\\
\rho_\ast \fExt^k_{\cA_{X}^{loc}}(\cM, \cA_{X}^{loc}) \simeq \fExt^{k-n}_{\cO_{M}^{\hbar, loc}}(\rho_\ast\cM, \cO_{M}^{\hbar,loc}) \label{for:loc}.
\end{align}
where $n=d_X/2=d_M$.
\end{prop}

\begin{proof}
We only prove formula \eqref{for:pasloc}, the proof of formula \eqref{for:loc} being similar. Applying Theorem \ref{thm:dualite} with $X_1=M$, $X_2= X$ and $X_3= \pt$, $\cK_1= \delta_{\rho \ast}\cA_{{X}}$ and $\cK_2=\cM \in \Der_{\coh,\rho-\fin}(\cA_{X})$ and using lemma \ref{lem:operation}, we get that
\begin{align*}
\Du^\prime_{M}(\rho_{\ast} \cM)&\simeq \Du^\prime_{M}(\delta_{\rho \ast} \cA_{X} \conv[X] \cM)\\ 
& \simeq \Du^\prime_{{M \times {X}^a}}(\delta_{\rho \ast} \cA_{X}) \conv[X^{a}] \omega_{X^{a}} \conv[X^{a}] \Du^\prime_{X}(\cM)\\
& \simeq \w_{M^a}^{-1} \conv[\cA_{M^a}] \w_{M^{a}} \conv[\cA_{M^a}] \Du^\prime_{{M \times {X}^a}}(\delta_{\rho \ast} \cA_{X}) \conv[X^{a}] \Du^\prime_{X}(\cM)[d_X]\\
& \simeq \w_{M^{a}} \conv[\cA_{M^a}] \Du^\prime_{{M \times {X}^a}}(\delta_{\rho \ast} \cA_{X}) \conv[X^{a}] \Du^\prime_{X}(\cM)[d_X-d_M]\\
& \simeq \rho_\ast \Du^\prime_{X}(\cM)[d_X-d_M].
\end{align*}

Taking the cohomology and using the fact that $\rho: \Supp(\cM) \to M$ is finite, we get that
\begin{equation*}
\rho_\ast \fExt^k_{\cA_{X}}(\cM, \cA_{X})) \simeq \fExt^{k-n}_{\cO_{M}^{\hbar}}(\rho_\ast\cM, \cO_{M}^{\hbar}).
\end{equation*}%
\end{proof}

We will need the following result (see \cite{BR}, \cite{Pop}).
\begin{thm}[Popescu, Bhatwadekar, Rao] \label{thm:free}
Let $R$ be a regular local ring, containing a field, with maximal ideal $\mathfrak{m}$ and $t \in \mathfrak{m} \setminus \mathfrak{m}^2$. Then every finitely generated projective module over the localized ring $R_t=R[t^{-1}]$ is free.
\end{thm}

As a consequences of Proposition \ref{prop:dual} and Theorem \ref{thm:free}, we get the analogue for DQ-modules of \cite[Proposition 4.6]{KV}.

\begin{prop}\label{prop:locibre}
Let $\cM \in \Mod_{\coh}^{\rho-\fin}(\cA_{X}^{loc})$. The sheaf $\rho_\ast \cM$ is locally free over $\cO_{M}^{\hbar,loc}$.
\end{prop}

\begin{proof}
The question is local on $M$ and $\rho$-finite coherent modules are locally good on the base. Then there exists an open covering $(U_i)_{i \in I}$ of $M$ such that for every $i \in I$, $\cM|_{\rho^{-1}(U_i)}$ is good. Once, one has restricted $\cM$ to such an $U_i$ the proof is exactly the same as in \cite[Proposition 4.6]{KV}. We include it for the sake of completeness. 

By Lemma \ref{lem:finhol}, $\cM$ is holonomic. Then
\begin{equation*}
\fExt_{\cA_{X}^{loc}}^k(\cM,\cA_{X}^{loc})=0 \quad \textnormal{for} \; k \neq \frac{d_X}{2}.
\end{equation*}

It follows from Proposition \ref{prop:dual} that
\begin{equation*}
\fExt_{\cO_{M}^{\hbar,loc}}^k(\rho_\ast\cM,\cO_{M}^{\hbar,loc})=0 \quad \textnormal{for} \; k \neq 0.
\end{equation*}

By taking the stalk at $m \in M$, we get that
\begin{equation}\label{eq:vanishing}
\Ext_{\cO_{M,m}^{\hbar,loc}}^k(\rho_\ast\cM_{m},\cO_{M,m}^{\hbar,loc})=0 \quad \textnormal{for} \; k \neq 0.
\end{equation}

The ring $\cO_{M,m}^{\hbar,loc}$ is a Noetherian ring with finite global dimension. This combined with equation (\ref{eq:vanishing}) imply that $\rho_\ast \cM_m$ is a finitely generated projective module.

It follows from Theorem \ref{thm:free} that $\rho_\ast \cM_{m}$ is a finitely generated free module. Since $\rho_\ast \cM$ is a coherent $\cO_{M}^{\hbar,loc}$-module, this implies that $\rho_\ast \cM$ is a locally free sheaf.    
\end{proof}

\section{Geometric preliminary}

Let $X$ be a complex manifold and $Z$ be a closed analytic subset of $X$ and $p \in X$. We denote by $C_p(Z)$ the tangent cone of $Z$ in $X$ at $p$. We refer the reader to \cite{whitney} and \cite{KS1} for a thorough study of the notion of tangent cone and to \cite{Mum} for a gentle introduction.

The following result can be found in \cite[Ch. III Proposition 1.1.3]{Sch}.

\begin{prop}\label{prop:coneanalytique}
Let $X$ be a complex manifold, let $Z$ be a closed analytic subset of $X$ and let $p$ be a point of $X$. The set $C_p(Z)$ is a closed analytic (even algebraic) subset of $T_p X$, conic for the action of $\C^{\ast}$ on $T_p X$.
\end{prop}

\begin{lemme}[{\cite[Lemma 8.11]{whitney}}]\label{lem:dimcone}
Let $X$ be a complex manifold, let $Z$ be a closed analytic subset of $X$ and $p \in X$. Then $\dim C_p(Z)=\dim_p Z$. If Z is of constant dimension near p then, $C_p(Z)$  is of constant dimension.
\end{lemme}
The next proposition is a special case of \cite[Proposition 3.2.3]{Herrman}.
\begin{prop}\label{prop:conefini}
Let $(X,p)$ be a germ of analytic space of dimension $d_X$, $(Y,q)$ be a germ of complex manifold of dimension $d_X$ and $f\colon(X,p) \to (Y,q)$ be a germ of analytic map. If $d_p f : (C_p(X),0) \to (T_q Y,0)$ is finite, then $f\colon (X,p) \to (Y,q)$ is finite.
\end{prop}
We will need the following result to prove Proposition \ref{prop:coneiso}.

\begin{lemme}[{\cite[Lemma 7.3]{KSsimple}}]\label{lem:annulationform}
Let $X$ be a complex manifold, $Y$ be a closed complex subvariety of $X$ and $f:X \to \C$ be a holomorphic function. Set $Z:=f^{-1}(0), Y':=\overline{(Y \setminus Z)}\cap Z$. Consider a $p$-form $\eta$, a $(p-1)$-form $\theta$ on $X$ and set
\begin{equation*}
\w=df \wedge \theta+ f \eta.
\end{equation*}
Assume that $\omega|_Y=0$. Then $\theta|_{Y'}=0$ and $\eta|_{Y'}=0$.
\end{lemme}

\begin{prop}\label{prop:coneiso}
Let $(X, \omega)$ be a symplectic complex manifold, $\Lambda$ be a closed isotropic analytic subset of $X$ and $p \in \Lambda$. The tangent cone $C_p(\Lambda)$ of $\Lambda$ at $p$ is an isotropic subset of the symplectic vector space $(T_p X,\omega_p)$.
\end{prop}

\begin{proof} We adapt the proof of \cite[Proposition 7.1]{KSsimple} to our needs. Let $(x;u)$ be a local symplectic coordinate system on $X$ such that
\begin{align*}
p=(0;0), && \omega_p=\sum_{i=1}^n du_i \wedge dx_i.
\end{align*}
Consider the deformation to the normal cone $\tilde{X}_p $ of $X$ along $p$. We have the following diagram
\begin{equation*}
\xymatrix{X & \ar[l]_-{q}  \tilde{X}_p \ar[r]^-{t} & \C\\
p \ar@{^{(}->}[u]& \ar[l]_-{q} \ar@{^{(}->}[u] q^{-1}(p).&
}
\end{equation*}
Let $(x;\xi,t)$ be the coordinate system on $\tilde{X}_p$. Then
\begin{equation*}
\begin{array}{l}
q(x;\xi,t)=(x;t \xi),\\
T_p X \simeq \lbrace (x;\xi,t) \in \tilde{X}_p | t=0 \rbrace,\\
C_p(\Lambda) \simeq \overline{q^{-1}(\Lambda) \setminus t^{-1}(0)}\cap t^{-1}(0).
\end{array}
\end{equation*}
The last isomorphism is a direct consequence of the sequential description of the tangent cone. 

Taking the pull-back of $\omega$ by $q$, we get
\begin{equation*}
q^\ast\omega= \sum_{i=1}^n d(t \xi_i) \wedge dx_i=dt \wedge (\sum_{i=1}^n \xi_i \wedge dx_i)+ t \sum_{i=1}^n d\xi_i \wedge dx_i.
\end{equation*}

The form $q^\ast \omega$ vanishes on $(q^{-1} (\Lambda))_\reg$. Hence, by Lemma \ref{lem:annulationform} the form $\sum_{i=1}^n d\xi_i \wedge dx_i$ vanishes on $C_p(\Lambda)$ which proves the claim. 
\end{proof}

\begin{prop}[{\cite[Proposition 1.6.1]{KasKai}}]\label{prop:laglin}
Let $(E,\omega)$ be a symplectic vector space and let $\Lambda$ be an isotropic homogeneous analytic subset of $E$. Then there is a Lagrangian linear subspace $\lambda$ such that $\Lambda \cap \lambda \subset \lbrace 0 \rbrace$.
\end{prop}

\begin{cor}\label{prop:interlag}
Let $(E,\omega)$ be a symplectic complex vector space of finite dimension. Let $\Lambda$ be Lagrangian analytic subset of $E$ and let $p$ be a point of $\Lambda$. Then there exists a linear Lagrangian subspace $\lambda$ of $E$ such that $\lambda \cap C_p(\Lambda)\subset \lbrace 0 \rbrace$.
\end{cor}

\begin{prop}\label{prop:mapfinie}
Let $(E, \omega)$ be a complex symplectic vector space of finite dimension, let $U$ be an open subset of $E$, let $\Lambda$ be a closed Lagrangian analytic subset of $U$ and let $p \in \Lambda$. There exists a linear Lagrangian subspace $\lambda$ of $E$, an open neighbourhood $V$ of $p$ and an open neighbourhood $M$ of $\pi_\lambda(p)$ the image of $p$ by the canonical projection $\pi_\lambda:E \to E/\lambda$, such that the map 
\begin{equation*}
\rho:=\pi_\lambda|_V^{M}:V \to M
\end{equation*}
 is finite when restricted to $\Lambda \cap V$. 
\end{prop}

\begin{proof}
Since $U$ is an open subset of a $E$, we identify $T_pU$ and $E$. 
It follows from Proposition \ref{prop:coneanalytique} that $C_p(\Lambda)$ is an algebraic subset of $E$ conic for the action of $\C^\ast$ on $E$. By Corollary \ref{prop:interlag}, there exist a Lagrangian linear subspace $\lambda$ of $E$ such that $C_p(\Lambda) \cap \lambda \subset \lbrace 0 \rbrace$. The quotient map $\pi_\lambda : E \to E /  \lambda$  induces a germ of analytic map $\rho:(\Lambda,p) \to (E/\lambda,  \pi_\lambda(p))$. We consider the differential $d_p\pi_{\lambda}:(C_p(\Lambda),0) \to (E/\lambda,0)$. Since $d_p \rho=\pi_\lambda$, we have 
\begin{equation*}
(d_p \rho)^{-1}(0)=\pi^{-1}_\lambda(0) \cap C_p(\Lambda)\subset\lbrace 0 \rbrace.
\end{equation*}

Thus, the germ of map $d_p \rho:(C_p(\Lambda),0) \to (E/\lambda,0)$ is finite and $\dim_p \Lambda = \dim E / \lambda$ which implies that the germ of analytic map $\rho:(\Lambda,p) \to (E/\lambda,  \pi_\lambda(p))$ is finite.
\end{proof}

\section{The codimension-three conjecture}

As indicated in the introduction the strategy to prove the codimension-three conjecture for DQ-modules is to reduce the non-commutative statement to a commutative problem where it is possible to apply Theorem \ref{thm:codimcom} due to Kashiwara and Vilonen.

\subsection{Proof of the Theorem \ref{thm:codim}}

The aim of this subsection is to prove the codimension-three conjecture for DQ-modules. Let us recall Theorem \ref{thm:codim}.

\begin{theoetoile}
Let $X$ be a complex manifold endowed with a DQ-algebroid stack $\cA_X$ such that the associated Poisson structure is symplectic. Let $\Lambda$ be a closed Lagrangian analytic subset of $X$, and $Y$ be a closed analytic subset of $\Lambda$ such that $\codim_\Lambda Y \geq 3$. Let $\cM$ be a holonomic $(\cA^{loc}_X|_{X \setminus Y})$-module, whose support is contained in $\Lambda \setminus Y$. Assume that $\cM$ has an $\cA_X|_{X \setminus Y}$-lattice. Then $\cM$ extends uniquely to a holonomic module defined on $X$ whose support is contained in $\Lambda$. 
\end{theoetoile}
Let $X$, $\Lambda$, $Y$ and $\cM$ as in Theorem \ref{thm:codim}. We write $j: X \setminus Y \to X$ for the open embedding of $X \setminus Y$ into $X$.
\subsubsection{Unicity of the extension}
We prove the unicity of the extensions of $\cM$ to $X$. As in \cite{KV}, we have
\begin{prop}\label{prop:unicite}
Let $\cM^\prime$ be an extension of $\cM$.
Then $\cM^\prime$ is isomorphic to $j_\ast \cM$. 
\end{prop}

\begin{proof}
Let $\cM^\prime$ be an extension of $\cM$. We have the following exact sequence.
\begin{equation*}
0 \to \Hn_Y^0(\cM^\prime) \to \cM^\prime \to j_\ast j^{-1} \cM^\prime \to \Hn_Y^1(\cM^\prime) \to 0.
\end{equation*}
By definition of an extension, $j^{-1} \cM^\prime \simeq \cM$. The above exact sequence becomes
\begin{equation*}
0 \to \Hn_Y^0(\cM^\prime) \to \cM^\prime \to j_\ast \cM \to \Hn_Y^1(\cM^\prime) \to 0.
\end{equation*}
Since, $\codim_\Lambda Y=3$, it follows by Lemma \ref{lem:vanishing} that $\Hn_Y^i(\cM^\prime)\simeq 0$ for $i=0,1$. This implies that $\cM^\prime \simeq j_\ast \cM$.
\end{proof}

\subsubsection{The support of the extension is Lagrangian}
As in \cite{KV}, we have

\begin{prop}
The support of $j_\ast \cM$ is a closed Lagrangian analytic subset of $X$.
\end{prop}

\begin{proof}
The module $\cM$ is holonomic. Thus, $\Supp(\cM)$ is a complex analytic Lagrangian subset of $X \setminus Y$ and $\Supp(j_\ast \cM)=\overline{\Supp(\cM)}$. By the Remmert-Stein theorem the set $\overline{\Supp(\cM)}$ is a closed analytic subset of $X$ of dimension $d_X / 2$ and is isotropic since the closure of an isotropic subset is isotropic.
\end{proof}
\subsubsection{Proof of the coherence}\label{proofcoherence}
The only thing that remain to prove is that the $\cA_X^{loc}$-module $j_\ast \cM$ is coherent. We will need the following proposition.
\begin{prop}\label{prop:locom}
Let $X$ be a complex manifold and $Y$ be a closed analytic subset of $X$ and $j:X \setminus Y \to X$ be the inclusion of $X \setminus Y$ into $X$. If $\cN$ is a coherent reflexive $\cO_{X \setminus Y}^\hbar$-module. Then
\begin{equation*}
(j_\ast \cN)^{loc} \simeq j_\ast (\cN^{loc}).
\end{equation*}
\end{prop}

\begin{proof}
By adjunction between $j^{-1}$ and $ j_\ast$, we have  
\begin{equation*}
j^{-1} j_\ast \cN \to\cN.
\end{equation*}
Then
 \begin{equation*}
j^{-1}(\C_{X \setminus Y}^{\hbar,loc}\te_{\C^\hbar_{X \setminus Y}} j_\ast \cN) \to \C_X^{\hbar,loc}\te_{\C^\hbar_X} \cN.
\end{equation*}
 Finally we get a canonical map $\phi:(j_\ast \cN)^{loc} \to j_\ast(\cN^{loc})$.

Since $\cN$ is reflexive, it has no torsion and in particular no $ \hbar$-torsion. Thus, the morphism $ \phi$ is a monomorphism. 

The morphism $\phi$ is an epimorphism. Indeed, let $U$ be a connected open subset of $X$ and $u \in (j_\ast(\cN^{loc}))(U)$. Locally on $U \setminus Y$, there exist $n \in \N$ such that $\hbar^n u \in \cN$. Let us show that this is true globally.

Let $\phi_n:\cN ^{loc} \to\cN^{loc} / \hbar^{-n} \cN$ with $n \in \N$. Assume that there exist $x \in U\setminus Y$ such that $\phi_n(u)_x=0$. Set $v=\phi_n(u)$. Let us show that $\supp(v)$ is open in $U \setminus Y$. There is the following isomorphism
\begin{equation*}
\cN^{loc} / \hbar^{-n} \cN \simeq \bigcup_{m \in \N} \hbar^{-m} \cN / \hbar^{-n} \cN.
\end{equation*}
Thus, locally $v$ belongs to some $\hbar^{-m} \cN / \hbar^{-n} \cN$. Moreover, $\hbar^{-m} \cN / \hbar^{-n} \cN$ are $\cO_{X \setminus Y}$-modules without $\cO_{X \setminus Y}$-torsion this implies that the support of their sections is open. Thus, the support of $v$ is a union of open sets which implies that $\supp(v)$ is open. The set $Y$ is a closed analytic subset of $U$ different from $U$. Thus, $U \setminus Y$ is connected.  The support of $v$ is open and closed and is different form $U \setminus Y$  by assumption. Hence, we have $\supp(v)=\emptyset$. This implies that $u \in \hbar^{-n}\cN(U \setminus Y)$. It follows that $\phi$ is an epimorphism.
\end{proof}

We now address the question of the coherence of $j_\ast \cM$. Let us briefly recall the setting. We consider a complex symplectic manifold $X$, $\Lambda$ is a closed Lagrangian analytic subset of $X$ and $Y$ is a closed analytic subset of $\Lambda$ such that $\codim_\Lambda Y \geq 3$. The problem of showing that a certain DQ-module is coherent is a local problem. Thus, we just need to work in a neighbourhood $U$ of a point $p \in Y$. As pointed out in \cite{KV}, we can assume, by working inductively on the dimension of the singular locus of $Y$ that $p$ is a smooth point of $Y$.

Shrinking $U$ if necessary, we can assume by Darboux's theorem that $U$ is an open subset of a symplectic vector space $(E,\omega)$.  Then by Proposition \ref{prop:mapfinie}, we know that there exists a Lagrangian linear subspace $\lambda$ of $E$, an open neighbourhood $V$ of $p$ in $U$ and an open neigbhourhood $M$ of $\pi_\lambda(p)$ in $E / \lambda$  such that the map $\rho:V \to M$ is finite when restricted to $\Lambda \cap V$. By shrinking $U$, we can assume that $U=V$ and that $Y$ is smooth. 
As $\rho$ is an open map we can also assume that $M=\rho(U)$. We set $N=\rho(Y)$, then the map
\begin{equation*}
\rho|_Y:Y \to N
\end{equation*}
is an isomorphism and $\codim N \geq 3$. We denote by $\rho^\prime$ the restriction of $\rho$ to $U \setminus \rho^{-1}(N)$, that is
\begin{equation*}
\rho^\prime: U \setminus \rho^{-1}(N) \to M \setminus N.
\end{equation*}

Writing $i:U \setminus \rho^{-1}(N) \hookrightarrow U \setminus Y$ for the inclusion, we have the following Cartesian square
\begin{equation*}
\xymatrix{
U \setminus \rho^{-1}(N) \ar@{^{(}->}[r]^-{j\circ i} \ar[d]_-{\rho^\prime} \ar@{}[rd]|{\square} & U \ar[d]^-{\rho}\\
M \setminus N \ar@{^{(}->}[r]_-{j^\prime} & M.
}
\end{equation*}

Since $\lambda$ is a linear Lagrangian subspace of $E$, we can find a symplectic basis $(e_1,\ldots,e_n;f_1,\ldots,f_n)$ of $E$ such that $f_i \in \lambda$ for $1 \leq i \leq n$. Let $(x,u)$ be the coordinates system associated to this basis. Then $\rho(x,u)=x$. Finally, we can assume that the star-algebra on $U$ is the Moyal-Weil star-algebra since on a symplectic variety all the star-algebras are locally isomorphic. Then we are in the setting of subsection \ref{subsec:cohDQ} and section \ref{sec:duality}.

We set $Y^\prime=\rho^{-1}(N)$ and we consider the restriction of $\cM$ and $\cN$ to $U \setminus Y^\prime$ and still write $\cM$ and $\cN$ for their restriction. By assumption $\cM$ has a lattice $\cN$ on $U\setminus Y^\prime$. Since $\cM$ is holonomic, it follows from Proposition \ref{prop:hollattice} that the module 
\begin{equation*}
\cN^\prime:=\fExt^n_{\cA_{U \setminus Y^\prime}}(\fExt^n_{\cA_{U \setminus Y^\prime}}(\cN,\cA_{U \setminus Y^\prime}),\cA_{U \setminus Y^\prime})
\end{equation*}
with $n=d_X/2$ is a lattice of $\cM$. Moreover, using Proposition \ref{prop:dual} we have the isomorphism
\begin{align*}
\rho^\prime_\ast \cN^\prime &\simeq \rho^\prime_\ast \fExt^n_{\cA_{U \setminus Y^\prime}}(\fExt^n_{\cA_{U \setminus Y^\prime}}(\cN,\cA_{U \setminus Y^\prime}),\cA_{U \setminus Y^\prime}) \\
                            &\simeq \fExt^0_{\cO_{M \setminus N}^\hbar}(\fExt^0_{\cO_{M \setminus N}^\hbar}(\rho^\prime_\ast \cN,\cO_{M}^\hbar),\cO_{M}^\hbar). 
\end{align*}
Then $\rho^\prime_\ast \cN^\prime$ is the dual of a sheaf. Thus, it is reflexive since the dual of a sheaf is always reflexive. By Proposition \ref{prop:finitness}, $\rho^\prime_\ast \cN^\prime$ is a coherent $\cO_{M \setminus N}^\hbar$-module. Applying Theorem \ref{thm:codimcom}, we obtain that $j_\ast^\prime \rho_\ast^\prime \cN^\prime$ is a coherent $\cO_{M}^\hbar$-module. By Proposition \ref{prop:locom}, $(j^\prime_\ast \rho_\ast^\prime \cN)^{loc}\simeq j^{\prime}_\ast((\rho_\ast^\prime \cN)^{loc})$ and using the projection formula, we have
\begin{align*}
\cO_{M \setminus N}^{\hbar,loc} \te_{\cO_{M \setminus N}^{\hbar}} \rho^\prime_\ast \cN^\prime & \simeq \rho^\prime_\ast ( \cA_{U \setminus Y^\prime}^{loc} \te_{\cA_{U \setminus Y^\prime}} \cN^\prime)\\
             & \simeq \rho^\prime_\ast \cM. 
\end{align*}
We have just shown that the $\cB_M$-module $j_\ast^\prime \rho^\prime_\ast \cN^\prime$ is a $\cO_M^\hbar$-lattice of the $\cO_M^{\hbar,loc}$-module $j_\ast^\prime \rho^\prime_\ast \cM$ and it is also a $\cB_{M}$-submodule of $j_\ast^\prime \rho^\prime_\ast \cM$. Moreover, we have the following commutative diagram
\begin{equation*}
\xymatrix{
\rho_\ast (j\circ i)_\ast \cM \ar[r]^-{\sim} & j_\ast^\prime \rho^\prime_\ast \cM\\ 
\rho_\ast (j\circ i)_\ast \cN^\prime \ar@{^{(}->}[u] \ar[r]^-{\sim} & j_\ast^\prime \rho^\prime_\ast \cN^\prime. \ar@{^{(}->}[u]
}
\end{equation*}  

It follows that $\rho_\ast (j \circ i)_\ast \cM$ has a coherent $\cO_{M}^\hbar$-lattice which is also a $\cB_{M}$-submodule of $\rho_\ast (j \circ i)_\ast \cM$. By Proposition \ref{prop:finitness} (ii) (b), it follows that $(j\circ i)_\ast \cM$ is coherent.

We consider the analytic set $Y''=Y' \cap \Lambda$. Since, the restriction of $\rho$ to $\Lambda$ is a finite holomorphic map it follows that $\codim_\Lambda Y'' \geq 3$. We denote by $l: U\setminus Y''\to U\setminus Y$ the inclusion. We now remember that we had implied the restriction of $\cM$ to $U\setminus Y^\prime$ i.e. that $\cM$ stands for $i^{-1}\cM$.
First, notice that
\begin{equation*}
(j \circ l)^{-1} j_\ast i_\ast i^{-1 } \cM \simeq l^{-1} i_\ast i^{-1} \cM \simeq l^{-1} \cM.
\end{equation*} 
Then Proposition \ref{prop:unicite} implies that
\begin{equation*}
j_\ast i_\ast i^{-1} \cM \simeq j_\ast l_\ast l^{-1} \cM.
\end{equation*}

We claim that $l_\ast l^{-1} \cM \simeq \cM$. Indeed, $Y^{\prime\prime} \setminus Y$ is an analytic subset of $U \setminus Y$ such that $\codim_{\Lambda \setminus Y} (Y^{\prime\prime} \setminus Y) \geq 3$. The module $l^{-1} \cM$ satisfies the hypothesis of Theorem \ref{thm:codim}, moreover $\cM$ is an extension of $l^{-1} \cM$ to $U \setminus Y$. It follows from Proposition \ref{prop:unicite} that $\cM \simeq l_\ast l^{-1} \cM$. This implies that $j_\ast \cM \simeq j_\ast l_\ast l^{-1} \cM$. This proves that $j_\ast \cM$ is a holonomic $\cA_X^{loc}$-module.

\subsection{Proof of the Theorem \ref{thm:subcodim}}
The aim of this section is to prove Theorem  \ref{thm:subcodim}. As explained in \cite{KV}, it is easier to prove an equivalent version of this statement for quotient of a given module. We state this version below.

\begin{thm}
Let $X$ be a complex manifold endowed with a DQ-algebroid stack $\cA_X$ such that the associated Poisson structure is symplectic. Let $\Lambda$ be a closed Lagrangian analytic subset of $X$, and $Y$ be a closed analytic subset of $\Lambda$ such that $\codim_\Lambda Y \geq 2$. Let $\cM$ be a holonomic $(\cA_X^{loc})$-module whose support is contained in $\Lambda$ and let $\cM_2$ be a $\cA_{X}^{loc}|_{X \setminus Y}$-module which is a quotient of $\cM|_{X \setminus Y}$. Then $\im(\cM \to j_\ast \cM_2)$ is a coherent $\cA_X^{loc}$-module.
\end{thm}

\begin{proof}
The proof is formally the same as the proof of \cite[Theorem 5.1]{KV}. All the results needed for the DQ-modules version have been proved earlier in this paper.
\end{proof}

\section{A conjecture of Pierre Schapira}
 DQ-modules are defined on any complex Poisson manifolds. The two extreme cases of Poisson brackets are on one hand the trivial bracket and on the other hand the case of non-degenerate Poisson brackets that is to say the case of symplectic manifolds. Furthermore, it appears from the proofs of the codimension-three conjecture for holonomic DQ-modules and holonomic microdifferential modules that holonomicity should be thought as a shifted version of reflexivity. Indeed, if $\cM$ is an holonomic DQ-module over a DQ-algebra $\cA_X$, then 
\begin{equation*}
\cM \stackrel{\sim}{\to} \fExt_{\cA^{loc}_X}^n(\fExt_{\cA^{loc}_X}^n(\cM, \cA^{loc}_X),\cA^{loc}_X).
\end{equation*} 

This has led to the following conjecture which extends the codimension-three conjecture for DQ-modules on symplectic manifolds to DQ-modules on any complex Poisson manifolds.

\begin{defi}
Let $(X,\cA_X)$ be a complex manifold endowed with a DQ-algebroid $\cA_X$ and $\cM$ a coherent $\cA_X$-module. We say that $\cM$ is $d$-reflexive if
\begin{enumerate}[(i)]
\item $\cM$ has no $\hbar$-torsion,
\item $\fExt_{\cA_X}^j(\cM,\cA_X)=0$ for $j < d$,
\item $\cM \stackrel{\sim}{\rightarrow} \fExt_{\cA_X|_U}^d(\fExt_{\cA_X}^d(\cM,\cA_X),\cA_X)$.
\end{enumerate}
\end{defi}

\begin{conj}[(Schapira)]
Let $(X,\cA_X)$ be a complex manifold endowed with a DQ-algebroid $\cA_X$. Let $Y$ be a closed complex analytic subset of $X$ such that $\codim_X Y \geq d+3$ and denote by $j \colon X \setminus Y \hookrightarrow X$ the open inclusion of $X \setminus Y$ into $X$. Let $\cM$ be a coherent $\cA_X|_{X \setminus Y}$-module. If $\cM$ is $d$-reflexive, then $j_\ast\cM$ is a coherent $\cA_X$-module.
\end{conj}

The proof of the codimension-three conjecture for holonomic DQ-modules cannot be adapted to prove this statement. Indeed, the proof relies on the fact that locally on a complex symplectic manifold there is up to isomorphism only one star-algebra. A possible approach would be to, first prove the theorem for a first order associative deformation of $\cO_X$ and then to extend it to DQ-algebras by techniques similar to the one used in the proof of \cite[Theorem 1.6]{KV}.

There is also the following weaker conjecture
\begin{conj}[(Schapira)]
Let $(X,\cA_X)$ be a complex manifold endowed with a DQ-algebroid $\cA_X$. Let $Y$ be a closed complex analytic subset of $X$ such that $\codim_X Y \geq d+3$ and denote by $j \colon X \setminus Y \hookrightarrow X$ the open inclusion of $X \setminus Y$ into $X$. Let $\cM$ be a coherent $\cA^{loc}_X|_{X \setminus Y}$-module. Assume that $\cM$ has $\cA_{X \setminus Y}$-lattice and that $\fExt_{\cA^{loc}_X}^j(\cM,\cA^{loc}_X)=0$ for $j \neq d$. Then $j_\ast\cM$ is a coherent $\cA^{loc}_X$-module.
\end{conj}

\section{Appendix}

In this appendix, we collect facts, well-known to the specialists, concerning $p$-finite sheaves. We include them for the sake of completeness.

\begin{lemme}\label{lemme:projfinie}
Let $X$ and $Y$ be two first countable, Hausdorff spaces, $p:X \to Y$ be a continuous map, $U$ be an open subset of $X$ and $F$ be a closed subset of $U$ such that $p|_F$ is proper. Then $F$ is closed in $p^{-1}(p(U))$.
\end{lemme}

\begin{proof}
Let $(u_n)_{n \in \N}$ be a sequence of points of $F$ converging towards a point $l\in  p^{-1}(p(U))$. Since $p|_F \colon F \to Y$ is closed, $p(F)$ is a closed subset of $Y$. Thus, the limit $p(l)$ of the sequence $(p(u_n))_{n \in \N}$ belongs to $p(F)$. Consider   $K=\lbrace p(u_n) \rbrace_{n \in \N} \cup \lbrace p(l) \rbrace \subset p(F)$. It is a compact subset of $Y$ and $p|_F$ is proper thus $p|_F^{-1}(K)$ is compact and thus is closed in $p^{-1}(p(U))$. Moreover, for every $n \in \N, \; (u_n) \in p|_F^{-1}(K)$. Thus, $l \in F$. It follows that $F$ is closed in $p^{-1}(p(U))$.
\end{proof}

A direct consequence of the above result is the following statement.

\begin{cor}\label{cor:finitesheaf}
Let $X$ and $Y$ be two first countable Hausdorff spaces, $U$ an open subset of $X$, $p:X \to Y$ a continuous map and $i:U \hookrightarrow p^{-1}(p(U))$ the inclusion of $U$ into $p^{-1}(p(U))$. Let $\cF$ be a sheaf of Abelian groups on $U$ such that $p|_{\Supp(\cF)}$ is a finite map. Then $i_\ast \cF$ is a $p|_{p^{-1}(p(U))}$-finite sheaf. 
\end{cor}

\begin{lemme}\label{lem:zerostalk}
Let $X$ and $Y$ be two Hausdorff spaces. Let $p:X \to Y$ be a continuous map, $\cR$ be a sheaf of rings and $\cF$ be a $p$-finite sheaf of $\cR$-modules on $X$. 
If $y \in Y$ is such that $p^{-1}(y) \cap \Supp(\cF) =\emptyset$ then there exist an open neighborhood $V$ of $y$ such that  $(p_\ast \cF)|_V=0$.
\end{lemme}

\begin{proof}
Since $p|_{\Supp(\cM)}$ is finite, $p(\Supp(\cF))$ is closed in $Y$. Consider the open set $V={}^c p(\Supp(\cF))$. Then, $y \in V$ and $p^{-1}(V) \cap \Supp(\cF)=\emptyset$. It follows that $p_\ast \cF |_V=0$ which proves the claim.
\end{proof}

Recall the following elementary result concerning finite maps (see \cite[Ch 2 §3]{GR}).
\begin{thm}\label{thm:finitemap}
Let  $S$ be a Hausdorff space and let $f:S \to T$ be a finite map. Let $x_1, \ldots,x_t$ be the distinct points of a fiber $f^{-1}(y)$, $y \in f(S)$ and let $U^\prime_1, \ldots , U^\prime_t$ be pairwise disjoint open neighbourhoods of $x_1, \ldots ,x_t$ in $S$. Then any neighbourhood $V^\prime$ of $y$ contains an open neighbourhood $V$ of $y$ with the following properties:
\begin{enumerate}

\item $U_1:=U_1^\prime \cap f^{-1}(V), \ldots, U_t ^\prime:=U_t^\prime \cap f^{-1}(V)$ are pairwise disjoint open neighbourhoods of $x_1,\ldots,x_t$ in $S$.

\item $f^{-1}(V)= \bigcup_{i=1}^t U_j$; in particular $f(U_j) \subset V$ for all $j$.

\item All of the induced maps $f_{{U_j},V}:U_j \to V$ are finite.
\end{enumerate}
\end{thm}

\begin{prop}\label{prop:sheafdecompo}
Let $X$ and $Y$ be two Hausdorff spaces and $p:X \to Y$ be a continuous map. Let $\cR$ be a sheaf of rings and $\cF$ be a $p$-finite sheaf of $\cR$-modules on $X$. Let $x_1, \ldots,x_t$ be the distinct points of a fiber $p|_{\Supp(\cF)}^{-1}(y)$ with $y \in p|_{\Supp(\cF)}(\Supp(\cF))$ and let $U^\prime_1, \ldots , U^\prime_t$ be pairwise disjoint open neighbourhoods of $x_1, \ldots ,x_t$ in $X$ such that $\overline{U_i^\prime} \cap \overline{U_j^\prime}=\emptyset$ for $i \neq j$ and set $\Omega_i^\prime=U_i^\prime \cap \supp(\cF)$ for $1 \leq i \leq t$. Then any neighbourhood $V^\prime$ of $y$ contains an open neighbourhood $V$ of $y$ with the following properties:  
\begin{enumerate}
\item $U_i:=U_i^\prime \cap p^{-1}(V)$ and  $\Omega_i :=\Omega_i^\prime \cap p^{-1}(V)$ are pairwise disjoint open neighbourhoods of $x_1,\ldots,x_t$ respectively in $X$ and $S$.

\item $p^{-1}(V) \cap \supp(\cF) = \bigcup_{i=1}^t \Omega_j$ for $1 \leq j \leq t$.

\item All of the induced maps $p_{{\Omega_j},V}:\Omega_j \to V$ are finite.
\end{enumerate}
and if we denote by $\gamma_i : U_i \to p^{-1}(V)$ the inclusion of $U_i$ into $p^{-1}(V)$, there is an isomorphism of $\cR|_{p^{-1}(V)}$-module
\begin{equation*} \phi:\cF|_{p^{-1}(V)} \to \prod_{i=1}^{t} \gamma_{i \ast} \gamma_i^{-1} (\cF|_{p^{-1}(V)}).
\end{equation*}
\end{prop}

\begin{proof}
Let $\cF$ be a $p$-finite sheaf of $\cR$-modules. We write $S$ for the support of $\cF$. It is a closed subset of $X$. By assumption $p|_S:S \to Y$ is finite. Let $x_1, \ldots,x_t$ be the distinct points of a fiber $p|_S^{-1}(y)$, $y \in p(X)$ and let $U^\prime_1, \ldots , U^\prime_t$ be pairwise disjoint open neighbourhoods of $x_1, \ldots ,x_t$ in $X$. We set $ \Omega_i^\prime:= S \cap U_i^\prime$ for $1 \leq i \leq t$. Then by Theorem \ref{thm:finitemap}, there exists an open subset $V$ of $V^\prime$ such that
\begin{enumerate}
\item $U_i:=U_i^\prime \cap p^{-1}(V)$ and  $\Omega_i :=\Omega_i^\prime \cap p^{-1}(V)$ are pairwise disjoint open neighbourhoods of $x_1,\ldots,x_t$ respectively in $X$ and $S$.

\item $p^{-1}(V) \cap S= \bigcup_{i=1}^t \Omega_j$ for $1 \leq j \leq t$.

\item All of the induced maps $p_{{\Omega_j},V}:\Omega_j \to V$ are finite.
\end{enumerate}

Let $\gamma_i : U_i \to p^{-1}(V)$ be the inclusion of $U_i$ into $p^{-1}(V)$. The adjunction between $\gamma_{i \ast}$ and $\gamma_i^{-1}$ provides a natural map
\begin{equation*} 
 \cF|_{p^{-1}(V)} \to \gamma_{i \ast} \gamma_i^{-1} (\cF|_{p^{-1}(V)}).
\end{equation*}
This induces a map  
\begin{equation}\label{map:decomposition}
 \phi:\cF|_{p^{-1}(V)} \to \prod_{i=1}^{t} \gamma_{i \ast} \gamma_i^{-1} (\cF|_{p^{-1}(V)}).
\end{equation}
Let $x \in S \cap p^{-1}(V)$. Since $\overline{U_i^\prime} \cap \overline{U_j^\prime}= \emptyset$ for $i \neq j$, there exists a unique $j_0$ such that $x \in \Omega_{j_0}\subset U_{j_0}$. This implies that $\phi_x$ is an isomorphism. 

Now, assume that $x \in p^{-1}(V)$ does not belongs to $S \cap p^{-1}(V)$. Since we have $\Supp(\gamma_{i \ast} \gamma_i^{-1} (\cF|_{p^{-1}(V)}))=\overline{\Omega_i}^{p^{-1}(V)}$, it follows that 
\begin{equation*}
\cF_x\simeq  \prod_{i=1}^{t} \lbrack \gamma_{i \ast} \gamma_i^{-1}(\cF|_{p^{-1}(V)}) \rbrack_x \simeq 0. 
\end{equation*}
It results that $\phi_x$ is an isomorphism. Thus, $\phi$ is an isomorphism of sheaves.
\end{proof}

\end{document}